\documentclass[12pt]{article}
\usepackage[english]{babel}
\usepackage{amsmath,latexsym,amsthm,amsfonts}
\usepackage{amssymb,amscd}
\usepackage{amscd}
\usepackage{latexsym}
\usepackage{cite}

\usepackage{ifthen}
\usepackage[usenames]{color}

\numberwithin{equation}{section}

\DeclareMathOperator{\abs}{abs} \DeclareMathOperator{\Dif}{d}
\DeclareMathOperator{\const}{const}

\newcounter{Draft}
\usepackage{longtable}
\usepackage{fancyhdr}
\usepackage{amssymb}
\usepackage{multicol}
\usepackage{hyperref}

\usepackage[all]{xy}

\usepackage{graphicx}

\renewcommand{\leq}{\leqslant}
\renewcommand{\geq}{\geqslant}

\newtheorem{rem}{Remark}
\newtheorem{example}{Example}
\newtheorem{fact}{Fact}
\newtheorem*{fact*}{Fact}

\newcounter{IsEquat}
\AtBeginDocument{%

\let\origlabel\label

\renewcommand*{\label}[1]{%
\ifthenelse{\theDraft=1} { \ifthenelse{\theIsEquat=1}{
\underline{\textbf{\textcolor{red}{\ensuremath{#1}:}}} \nonumber
\\ } { \textcolor{blue}{label: }\textcolor{red}{\ensuremath{#1}:}
}\origlabel{#1} %
}{\origlabel{#1}}} }

\AtBeginDocument{%
\let\origBegin\begin%
\renewcommand*{\begin}[1]{%
\ifthenelse{\theDraft=1}%
{\ifthenelse{\equal{#1}{equation}}{\setcounter{IsEquat}{1}
\origBegin{eqnarray} \text{\textcolor{blue}{Formula: }} 
}{\origBegin{#1}}}{\origBegin{#1}} } }

\AtBeginDocument{%
\let\origEnd\end%
\renewcommand*{\end}[1]{%
\ifthenelse{\theDraft=1}%
{\ifthenelse{\equal{#1}{equation}}{\setcounter{IsEquat}{0}
\origEnd{eqnarray} }{\origEnd{#1}} }{\origEnd{#1}} } }

\title{Kepler's laws}
\author{M. Plakhotnyk}
\usepackage{fancyhdr}
\begin{document}
\maketitle

\begin{center}

\thispagestyle{empty}
{\LARGE Kepler's laws}\\
{\large with introduction to differential calculus}\\
{\large (book for scholars, who are interested in physics and
mathematics)}\\
{\Large prepared: M. Plakhotnyk}\\

{\Large makar.plakhotnyk@gmail.com}\\

\end{center}

\thispagestyle{empty}

\begin{abstract}
We explain the solution of the following two problems: obtaining
of Kepler's laws from Newton's laws (so called two bodies problem)
and obtaining the fourth Newton's law (the formula for
gravitation) as a corollary of Kepler's laws.

This small book is devoted to the scholars, who are interested in
physics and mathematics. We also make a series of digressions,
where explain some technique of the higher mathematics, which are
used in the proofs.
\end{abstract}

\newpage

\section{Introduction}

Both Newton's laws and Kepler's laws are classical facts, known
from the school course of physics. One of the beautiful facts in
physics is that Kepler's laws are equivalent to the Newton's law
of the gravitation.\\

Remind the formulation of Kepler's laws.\\
\textbf{Kepler's first law}: The orbit of a planet is an ellipse
with the Sun at one of the two foci.

\textbf{Kepler's second law}: A line segment joining a planet and
the Sun sweeps out equal areas during equal intervals of time.

\textbf{Kepler's third law}: The square of the orbital period of a
planet is proportional to the cube of the semi-major axis of its
orbit.\\

Remind the formulation of Newton's laws:\\
\textbf{Newton's first law}: if the vector sum of all forces
acting on an object is zero, then the velocity of the object is
constant.

\textbf{Newton's second law}: the constant force
$\overrightarrow{F}$, which acts on the body of the mass $m$,
produces the constant acceleration $\overrightarrow{a} =
\frac{\overrightarrow{F}}{m}$ in the direction of
$\overrightarrow{F}$.

\textbf{Newton's third law}:  all forces between two objects exist
in equal magnitude and opposite direction: if one object $A$
exerts a force $F_A$ on a second object $B$, then $B$
simultaneously exerts a force $F_B$ on $A$, and the two forces are
equal in magnitude and opposite in direction: $F_A = -F_B$.

\textbf{Newton's fourth law}: every point mass attracts every
single other other point mass by a force pointing along the line
intersecting both points. The force is proportional of the two
masses and inversely proportional to the square of the distance
between
them.\\

Notice, that 1-th and 3-rd Newton's law are technical. They are
formulated in the style of Euclidean axioms and thus, it is
impossible to try follow them from some complicated facts. We
consider the second Newton's law as the definition of the force,
considering the acceleration as the mathematical notion of the
second derivative.

Thus, the equivalence of the fourth Newton's law and Kepler's laws
is the unique non-trivial fact, which may exist as mathematical
theorem and it is.

The proof of the mentioned equivalence is not elementary. It uses
derivatives, integration, some basis of vector analysis, polar
coordinates. In this book we explain these notions as detailed, as
it is possible in the small book for scholars and our the main
deal is to show young people the powerfulness of mathematics and
inspirit them to study mathematics in more detailed, using other
books, dedicated not especially to Kepler's laws, or the
connection of mathematics and physics.

Moreover, we give in Section~\ref{sect-10} some non-standard
proves of two classical facts. In Section~\ref{sect-10a} we proof
``from physical reasonings'' the formulas for the derivative of
functions $\sin x$ and $\cos x$. In Section~\ref{sect-10b} we
prove ``by trivial reasonings'' the formula for the radius of the
curvature of the plain curve.

\newpage

\section{Derivatives and integrals}

\subsection{Physical approach}

\subsubsection{Derivatives}

The use of derivatives (precisely the invention of derivatives) in
the 17-th century completely transformed physics, because opened
the possibility for the complicated mathematically strict
calculations as the solutions of the physical (mechanical)
problems.

Kepler's laws were one of the first illustration of the
powerfulness of the new mathematical (and physical) notion of
derivative. From another hand, the derivative is not more then the
mathematical formalization of notion of velocity.

Suppose that a point $A$ moves on the line (say $X$-axis) and
denote $x=s(t)$ the position function of $A$, i.e. $s(t)$ is the
coordinate of $A$ at the time $t$. Then the classical notion of
the instantaneous velocity of $A$ will be some function $v(t)$,
dependent on $t$. The mathematical formalization of the
instantaneous velocity is exactly the derivative of the function
$s$ and it is traditionally denoted $v = s'$.

Fix moment of time $t_0$. The velocity of $A$ at this moment (the
instantaneous velocity) can be found as
\begin{equation}\label{eq-22} s'(t_0) \approx \frac{s(t_0+\Delta
t) -s(t_0)}{\Delta t},
\end{equation} where $\Delta t$ is ``small enough'', i.e. $\Delta
t\approx 0$, but $\Delta t \neq 0$. Moreover, the symbol $\approx$
is used in~\eqref{eq-22} only to notice that $\Delta t$ is not an
exact number, but is ``very small, but non-zero''. In fact, the
formula~\eqref{eq-22} defines some certain function, which is as
certain, as clear as the physical notion of instantaneous
velocity. Suppose that $dt$ is so small value of $\Delta t$,
that~\eqref{eq-22} can be considered as equality
\begin{equation}\label{eq-23}s'(t_0) = \frac{s(t_0 +dt) -
s(t_0)}{dt},
\end{equation} i.e. we can ignore the absolute value %
$\left|s'(t_0) - \frac{s(t_0 +dt) - s(t_0)}{dt}\right|$ assuming
that it is zero. Now denote $$ds(t_0) = s(t_0 +dt) - s(t_0)$$ and
rewrite~\eqref{eq-23} as \begin{equation}\label{eq-24}s'(t_0) =
\frac{ds(t_0)}{dt}.
\end{equation} %
Traditionally, people do not write $t_0$ in~\eqref{eq-24}, i.e.
left just \begin{equation}\label{eq-25}s' = \frac{ds}{dt},
\end{equation} %
but understand~\eqref{eq-25} with all the remarks before it.
Notice ones more, that $s$ and $s'$ in~\eqref{eq-25} are some
functions, and $t$ is their argument. Moreover, mathematics has a
powerful techniques how to find derivative $s'$ by given $s$.

\begin{example}\label{example-01}
Find the derivative of the function $s(t)=t^2$.
\end{example}

Plug $s(t)=t^2$ into~(\ref{eq-22}) and get $$ \frac{s(t+h)
-s(t)}{h} = \frac{(t+h)^2 - t^2}{h} = \frac{(t^2 +2th + h^2) -
t^2}{h} = 2t +h.
$$ Evidently, if $h\approx 0$, then $2t +h \approx 2t$, whence
$$
s'(t) =2t.
$$

We will need the following properties of the derivative for our
further reasonings.

1. For arbitrary function $f(x)$ and arbitrary constant $a$ the
formula $(af(x))' = a\cdot f'(x)$ holds. This formula follows
from~(\ref{eq-23}) as the definition of the derivative.

2. Let $f(x)$ and $g(x)$ be arbitrary functions. Denote $h(x) =
f(g(x))$. Then the derivative $h'(x)$ equals
\begin{equation}\label{eq-35}h'(x) = [f(g(x))]' = f'(g(x))\cdot g'(x).
\end{equation}

We will give at first an example of the use of~\eqref{eq-35} and
follow it by the proof.

\begin{example}
Find the derivative of $p(t)=t^4$.
\end{example}

Denote $s(t)=t^2$ and $p(t)=t^4$. We have obtained in
Example~\ref{example-01} that $s'(t)=2t$. Clearly, $p(t)=s(s(t))$.
For make our reasonings more clear, denote $f(t)=g(t)=t^2$ whence
$p(t)=f(g(t))$. The expression $f'(g(x))$ means that function $g$
should be plugged into $f'$, whence $f'(g(t)) = 2\cdot g(t) =
2t^2$. Thus, by~\eqref{eq-35}, $p'(t) = 2t^2\cdot 2t = 4t^3$.\\

\begin{proof}[Proof of~\eqref{eq-35}]
The derivative of $h$, which is written the by definition $$ h'(x)
= \frac{h(x+t) -h(x)}{t},\ t  \approx 0,
$$ can be rewritten as follows.%
$$\frac{h(x+t) -h(x)}{t} = \frac{f(g(x+t))
-f(g(x))}{g(x+t)-g(x)}\cdot \frac{g((x+t)-g(x)}{t},\, t  \approx
0. $$ %
The first fraction here is $f'$, where $g$ is plugged. The second
multiplier is $g'(x)$, whence~\eqref{eq-35} is proved.
\end{proof}

Also derivative $s'(t)$ is the tangent of the angle between
$x$-axis and the tangent-line for the graph $s(t)$ at the point,
corresponding to $t$. Indeed, denote $A(t,\, s(t))$ and
$B(t+\Delta t,\, s(t+\Delta t))$. Then the tangent of the line
$AB$ (which equals to the tangent of the angle $x$-axis and $AB$)
equals to the expression from formula~\eqref{eq-22}. If $\Delta t
\approx 0$, then $AB$ transforms to the tangent and~\eqref{eq-22}
becomes the derivative.

There are known a lot of methods of finding the derivatives in
mathematics. Precisely, derivatives of all the functions, which
are studied at school (i.e. $y=x^n$, $y=a^x$, $y=\sin x$, $y=\cos
x$, $y=\tan x$) can be found explicitly.

\subsubsection{Conservation laws}

Consider the following classical physical problem. Suppose that a
point of the mass $m$ moves with the velocity $v$ and constant
force $F$ acts on the point in the direction of the movement.

The following physical quantities are know in this case:

1. The kinetic energy of the point is $\frac{mv^2}{2}$.

2. The impulse of the point is $mv$.

3. If the force acts during the distance $l$, then the work of the
force is $Fl$.

4. If the force acts during the time $t$, then the impulse of the
force is $Ft$.\\

Moreover, is we ignore other factors, which cause the movement of
the body, then impulse of the force transforms to the impulse of
the body and also the work of the force transforms to the kinetic
energy of the body. In other words, the equality
\begin{equation}\label{eq-36}m\Delta v = F\Delta t
\end{equation} and \begin{equation}\label{eq-37}
\frac{m(\Delta v)^2}{2} = F\Delta l
\end{equation} hold. The equality~\eqref{eq-36} is called the
conservation of the impulse law and~\eqref{eq-37} is conservation
of the energy law.

School course of physics does not contain the clear explanation of
the notions of energy (kinetic energy) and impulse. In fact,
laws~\eqref{eq-36} and~\eqref{eq-37} are considered like
definition of these two notions and it is frequently said
that~\eqref{eq-36} and~\eqref{eq-37} are obtained
``experimentally''.

The same situation is about the notion of the ``force''. By second
Newton's law, the constant force $F$ causes the acceleration $a$
with the same direction as $F$ and such that
\begin{equation}\label{eq-38}F = am.\end{equation}
The precisely, formula~\eqref{eq-38} is not more than definition
of the force. If there is an acceleration, then~\eqref{eq-38}
defines something and this ``something'' is called force.
Moreover, the existence of the force should cause the
acceleration, again by~\eqref{eq-38}.

The acceleration is the derivative of the velocity, or, which is
the same, the ``second derivative'' of the position of the point
(i.e. the derivative of the derivative). It is better here to
understand the derivative just as ``new function, which is
constructed by some certain rulers from the former function''. In
any case, say that
\begin{equation}\label{eq-39}\left\{\begin{array}{l}
a = v'\, ;\\ a = s''\, .
\end{array}\right.\end{equation}

\begin{example}\label{example-02}
Second Newton's law implies the law of conversation of the
impulse.
\end{example}

Rewrite the 2-nd Newton's law~\eqref{eq-38} as $m\, \frac{dv}{dt}
= F,$ which is equivalent to $$m\, dv = F\, dt. $$ %
and is exactly the same as~\eqref{eq-36}.

\begin{rem}\label{rem-1}
Notice, that reasonings from the Example~\ref{example-02} can be
inverted and, thus, the following proposition can be obtained:
Suppose that material point moves on a line an the formula $mv =
Ft $ holds, where $F$ is some constant. Then this point moves with
constant acceleration $a = \frac{F}{m}$.
\end{rem}

\begin{example}
Suppose that material point moves on a line an the
formula~\eqref{eq-37} holds, where $F$ is some constant. Then this
point moves with constant acceleration, which can be found
by~\eqref{eq-38}.
\end{example}

\begin{proof}
Suppose that at some (former) moment of time the velocity and the
length equal zero, i.e. $v(0)=l(0)=0$. Then we can
rewrite~\eqref{eq-37} as $$ \frac{mv^2}{2} = Fl.
$$
Since left hand side and right hand side of the obtained equality
are equal, then their derivatives are equal too, i.e.
$$
\frac{m}{2}\cdot \, 2\, v\, v' = F\, l'.
$$ But $v'=a$ and %
$l'=v$, whence $mva = Fv,$ and, after cancellation, %
we get~\eqref{eq-38}.
\end{proof}

\begin{rem}\label{rem-2}
Notice, that reasonings from the Example~\ref{example-02} can be
inverted and, thus, the following proposition can be obtained:
Suppose that material point moves on a line an the formula
$\frac{mv^2}{2} = Fl $ holds, where $F$ is some constant. Then
this point moves with constant acceleration $a = \frac{F}{m}$.
\end{rem}

\subsection{Mathematical approach}

\subsubsection{Derivatives}

If a function $s=s(t)$ is given by quite simple formula, then
there are mathematical methods, which give us quite simple
formulas for the derivative $v= s'(t)$.

Suppose now, that we have a material point, which moves not just
on a line, but in a 3-dimensional space. It means that the
function $s =s(t)$ should be consisted of three functions, write
\begin{equation}\label{eq-01} s(t) = (x(t),\, y(t),\, z(t)).
\end{equation}
The derivative $s'$ of $s$ of the form~\eqref{eq-01} also is
defined by formulas~\eqref{eq-22}, \eqref{eq-23} and~\eqref{eq-24}
with correspond remarks, but now $s$ is a vector-function there,
whence the differences $s(t_0+\Delta t) -s(t_0)$, ${s(t_0+dt)
-s(t_0)}$ and $ds(t_0)$ are differences of vectors and, thus,
vectors.

\begin{fact}\label{fact-01}
If the function is given by~\eqref{eq-01}, then its derivative
equals $$ s'(t) = (x'(t),\, y'(t),\, z'(t)).
$$
\end{fact}

\begin{fact}\label{fact-09}
If~\eqref{eq-01} defines the position function of a point, then
the derivative $v(t) =s'(t)$ defines the velocity (which is,
clearly, a vector) and the derivative $a(t) = v'(t) = (s'(t))'$
defines the acceleration (which is a vector too). Notice, that in
this case $(s'(t))'$ is called \textbf{second derivative} and also
can be denoted as $s''(t)$ which means the same as $(s'(t))'$,
being the derivative of the derivative.
\end{fact}

\begin{fact}\label{fact-10}
If $g,\, h$ are functions $\mathbb{R}\rightarrow \mathbb{R}$, then
the derivative of the function $f(x) = g(h(x))$ can be found as $$
f'(x) = g'(t)\cdot h'(x),
$$ plugging $t = h(x)$. This fact is called \textbf{the chain rule} of the
differentiation.
\end{fact}

\begin{fact}\label{fact-11}
If $g,\, h$ are functions $\mathbb{R}\rightarrow \mathbb{R}$, then
the derivative of the function $f(x) = g(x)\cdot h(x)$ can be
found as
$$ f'(x) = g'(x)\cdot h(x) + g(x)\cdot h'(x).
$$ This fact is called \textbf{the product rule} of the
differentiation.
\end{fact}

\begin{fact}\label{fact-12}
\textbf{Table of derivatives}: \\
$\sin'x = \cos x$ and $\cos'x = -\sin x$.
\end{fact}

\begin{fact}\label{fact-05}
The derivative of the function $s(t)=\frac{1}{t}$ is $$ s'(t) =
\frac{-1}{t^2}.
$$
\end{fact}

\subsubsection{Integrals}

The derivative, which we have mentioned above, can be considered
now just as some rule of finding the new function $g = f'$ by the
given function $f$.

The converse problem also can be considered: by the given function
$g$ find the function $f$ such that $g = f'$. In this case the
function $g$ is called the \textbf{indefinite integral} of the
function $f$. The nature of this name will become evident after we
will understand the following fact.

Let a function $y=g(x):\, \mathbb{R}\rightarrow \mathbb{R}$ such
that $g(x)>0$ for all $x$ be given and let the graph of $g$ be
constructed. Fix any point $a\in \mathbb{R}$ and denote by
$A(t),\, t\geq a$ the area between the graph of $y=g(x)$ and
$X$-axis for $x\in [a,\, t]$ (see Fig.~\ref{fig-13}).

\begin{figure}[h]
\center{\includegraphics[width=8cm]{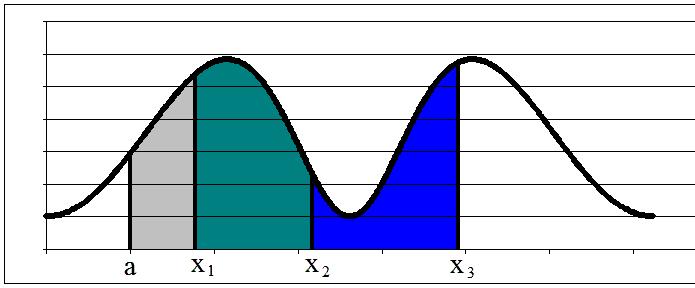}}
\caption{}\label{fig-13}
\end{figure}

It is easy to prove that $A'(a) = g(a)$. It follows from the
arbitrariness of $a$ that $A'(x) = g(x)$ for all $x\in
\mathbb{R}$. Suppose that (for the same $g$) the function $f$ is
such that $f' = g$. That $A' -f' =0$, which is the same as
$(A-f)'(x)=0$ for all $x\in \mathbb{R}$. In is easy to see that in
this case there exists a constant $c\in \mathbb{R}$ such that
\begin{equation}\label{eq-40}
f(x) = A(x) +c\end{equation} for all $x\in \mathbb{R}$,
whence~\eqref{eq-40} is the complete description of all the
functions $f$ such that $f'=g$. This fact is called the
\textbf{Newton-Leibnitz Theorem}.

From another hand, since the function $A(x)$  is defined as some
area, we can write it in ``a bit strange manner'', using the
function $g$, which in used for this area. Let $a\in \mathbb{R}$
as above and $x>a$ be fixed. Then geometrical figure (whose are we
are calculating to obtain $A(x)$) may be considered as consisted
of many-many rectangles of the small with $\Delta x$ (which will
become $\Dif x$ later) and of the height $g(s)$ (for each $s\in
[a,\, x]$). Denote $n = \frac{x-a}{\Delta x}$. Then the area
$A(x)$ can be calculated \begin{equation}\label{eq-44}A(x) \approx
\sum\limits_{k=0}^ng(a+k\Delta x)\Delta x.
\end{equation}

\begin{figure}[h]
\center{\includegraphics[width=8cm]{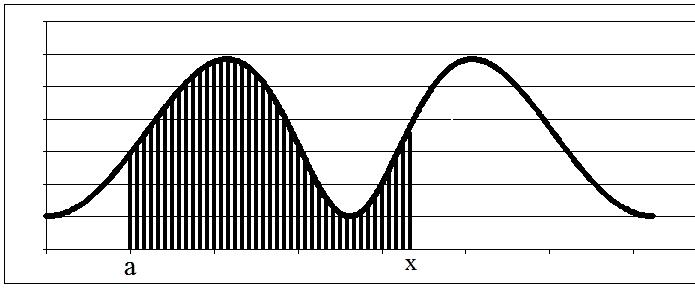}}
\caption{}\label{fig-14}
\end{figure}

When we decrease $\Delta x$ (in fact, change $\Delta x$ to the
differential $\Dif x$), then the limit value of the
expression~\eqref{eq-44} is denoted by
\begin{equation}\label{eq-45}A(x) = \int\limits_a^xg(s)\, \Dif s.
\end{equation} The expression~\eqref{eq-45} is called the \textbf{definite
integral} (of the function $g$ with boundaries $a$ and $x$). The
word ``integral'' here means the collection of ``infinitely many
rectangles of infinitely small width'' such that the sum of the
areal of these rectangles equals to the are $A(x)$.

\begin{figure}[h]
\center{\includegraphics[width=8cm]{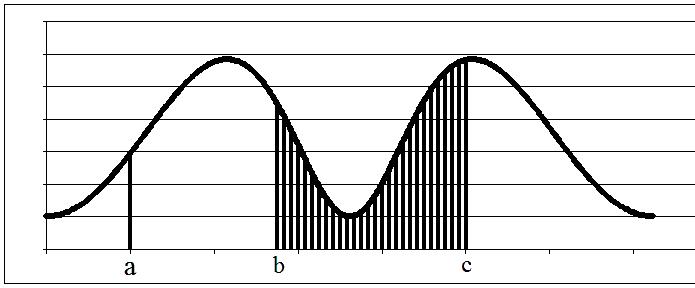}}
\caption{}\label{fig-15}
\end{figure}

For example, if we need to find the area ``between real points''
$b$ and $c$ (i.e. the ares, which is denoted at
Figure~\ref{fig-15}), then we calculate it as
\begin{equation}\label{eq-58}A = A(c) - A(b) \end{equation} and write $A =
\int\limits_b^cg(s)\, \Dif s.$ Notice, that it would not be good
enough if we will use $x$ in the notation $$ \int\limits_a^xg(x)\,
\Dif x
$$ in~\eqref{eq-45}, because $x$ in the upper bound of the integral
and $x$ in the expression $g(x)$ have the completely different
meaning. In the same time, we are free to write both $A =
\int\limits_b^cg(s)\, \Dif s$ or $A = \int\limits_b^cg(x)\, \Dif
x$ for the area~\eqref{eq-58}.

Thus, we can formulate the general rule for the calculating the
area $\int\limits_{x_1}^{x_2}g(t)\Dif t$:

1. Find the function $f$ such that $f' = g$.

2. The necessary area will be $f(x_2) -f(x_1)$~\footnote{This rule
is not ``completely correct''. It should be specified for some
``complicated situation'', i.e. if the function $f$ has vertical
asymptotes, or is braking. But we write this rule just for ``the
first acquaintance''.}.\\

We need to notice, that we meant ``not exactly the area'' under
the word ``area'' in this section. We have calculated the area by
the formula~\eqref{eq-44}, assuming that $\Delta x \approx 0$ but
with the additional assumption that $g(s)>0$ for $s\in [a,\, x]$.
If the function $g$ is negative at some (or all) points of $[a,\,
x]$, then~\eqref{eq-44} will, clearly, have negative summands and
the entire sum can also be negative. Such method of the
calculation of the area (precisely, such ``area'') is called the
the \textbf{signed area}. If we need ``real area'', we, evidently
should to calculate it as $\int\limits_a^x|g(s)|\Dif s$.

Notice, that the finding of the function $f$ (``finding'' here
means ``finding the explicit formula for $f$'') such that $f'=g$
(for the given $g$ needs some special techniques, which, clearly,
depends of the function $g$.

Suppose that we have an equation \begin{equation}\label{eq-65a}
\frac{f(x)}{g(y)} = \frac{\Dif y}{\Dif x},
\end{equation} where $y$ is considered as unknown function on $x$,
which has to be found, whenever functions $f$ and $g$ are given.
Notice, that~\eqref{eq-65a} is also called the
\textbf{differential equation of the separated variables}, because
it can be rewritten as $$ f(x)\Dif x = g(y)\Dif y,
$$ i.e. in the form where all the ``letters'' $x$ are left at one hand
of the equation and all the letter $y$ are put to another.

\begin{fact}\label{fact-16}
The solution of~\eqref{eq-65a} can be written in the form $$\int
f(x)\Dif x - \int g(y)\Dif y = C(x),$$ where $C'(x) = 0$.
\end{fact}

\subsection{Interesting theorems about derivatives and
integrals}\label{sect-10}

\subsubsection{Calculating of %
derivatives of $\sin x$ and $\cos x$ as simple corollary of the
physical interpretation of uniform movement by a
circle}\label{sect-10a}

Suppose that a point moves anti clock wise on a circle with radius
$1$ and center at the Origin. Clearly, the equation of the
movement will be $$ \left\{ \begin{array}{l}
x = \cos(\omega t +\omega_0)\\
y = \sin(\omega t +\omega_0),\end{array}\right.
$$ where $\omega$ is the radial speed and $\omega_0$ is the angle
at time $0$ between the radius vector of our point and the
$X$-axis. For the simplicity of the further reasonings assume that
$\omega_0 =0$.

Since the period of functions $\sin$ and $\cos$ is $2\, \pi$, then
the period of the rotation of our point can be calculated from the
equation $ \omega T = 2\, \pi, $ whence $$ T = \frac{2\,
\pi}{\omega}.
$$ Since the length of the circle is $2\, \pi$, then the
constant speed is $$ |\overrightarrow{v}| = \frac{2\, \pi}{T} =
\omega.
$$

Evidently the vector of the velocity is the tangent vector to the
circle, i.e. $\overrightarrow{v}$ is perpendicular to $s = (\cos
\omega t,\, \sin\omega t)$, whence $$ v = k(-\sin\omega t,\,
\cos\omega t)
$$ for some $k\in \mathbb{R}$.

From another hand, $|k| = \omega$, since $|\overrightarrow{v}| =
\omega$. Moreover, $k>0$, because the movement is anti clock wise.
Thus, $\overrightarrow{v} = \omega(-\sin\omega t,\, \cos\omega t)$
and this means that
$$
\left\{ \begin{array}{l}
\cos'\omega t = -\omega\, \sin\omega t,\\
\sin'\omega t = \omega\cos\omega t.\end{array}\right.
$$ plug $\omega =1$ and obtain formulas $$
\left\{ \begin{array}{l}
\cos't = -\sin t,\\
\sin't = \cos t.\end{array}\right.$$

\subsubsection{Simple proof of the formula for the %
radius of curvature for the plain curve}\label{sect-10b}

When we introduced the derivative, we looked at any curve as the
collection of line segments, whose direction changes from point to
point. Thus, each point, ``together with its neighbor point''
determine some line (or the direction of line). In the same manner
we can consider each point of the curve at ``two its neighbor
points''. These three points, being the vertices of the triangle
(if our curve does not contain line segments) will determine a
circle, which circumscribes this triangle.

We will need some mathematical facts.

\begin{fact}\label{fact-18}
Let $\overrightarrow{u} = (x_1,\, y_1)$ and $\overrightarrow{v} =
(x_2,\, y_2)$ be two non-collinear vectors on the plane $XOY$.
Denote $$ \mathcal{P}_3[\overrightarrow{u},\, \overrightarrow{v}]
= x_1y_2 - x_2y_1.
$$ Then the area of the parallelogram, whose sides are
$\overrightarrow{u}$ and $\overrightarrow{v}$ equals
$$A = |
\mathcal{P}_3[\overrightarrow{u},\, \overrightarrow{v}]|.
$$
\end{fact}

\begin{proof}
Denote $\theta$ the angle between $\overrightarrow{a}$ and
$\overrightarrow{b}$. By the formula for a dot-product we have
\begin{equation}\label{eq-20} x_ax_b + y_ay_b = ||a||\, ||b||\,
\cos\theta.
\end{equation}
By the formula for the area of a parallelogram we have
\begin{equation}\label{eq-21} A = ||a||\, ||b||\, \sin\theta.
\end{equation}
If follows from~\eqref{eq-20} and~\eqref{eq-21} that $$ A^2 =
||a||^2\, ||b||^2\, \left( 1- \frac{(x_ax_b + y_ay_b)^2}{||a||^2\,
||b||^2}\right) = (x_a^2 +y_a^2)\, (x_b^2 +y_b^2) - (x_ax_b +
y_ay_b)^2 =
$$$$= (\underline{x_a^2x_b^2} + x_a^2y_b^2
+x_b^2y_a^2 + \underline{\underline{y_a^2y_b^2}})%
- (\underline{x_a^2x_b^2} + 2x_ax_by_ab_y
+\underline{\underline{y_a^2y_b^2}}) =
$$$$ = x_a^2y_b^2
+x_b^2y_a^2 -2x_ax_by_ab_y =$$$$ = (x_ay_b - x_by_a)^2.$$

Let us have vectors $\overrightarrow{a} = (x_a,\, y_a)$ and
$\overrightarrow{b} = (x_b,\, y_b)$ and find the area of the
parallelogram, whose sides are $\overrightarrow{a}$ and
$\overrightarrow{b}$.
\end{proof}

Notice, that $[\overrightarrow{a},\, \overrightarrow{b}]$ is
called vector product of $\overrightarrow{u}$ and
$\overrightarrow{v}$ and is a vector $$ [\overrightarrow{u},\,
\overrightarrow{v}] = (0,\, 0,\, x_1y_2 - x_2y_1)
$$ and $\mathcal{P}_3$ is the projection of a vector to its third
coordinate, i.e. $$ \mathcal{P}_3(x,\, y,\, z) = z.
$$

We will explain some basic things about vector product in
Section~\ref{sect-09}. Nevertheless, we do not need the notion of
vector product now and we will be satisfied with
Fact~\ref{fact-18}.\\

The following properties of $\mathcal{P}_3[\overrightarrow{u},\,
\overrightarrow{v}]$ follow from the definition: $$
\mathcal{P}_3[\overrightarrow{u_1} +\overrightarrow{u_2},\,
\overrightarrow{v}] = \mathcal{P}_3[\overrightarrow{u_1},\,
\overrightarrow{v}] + \mathcal{P}_3[\overrightarrow{u_2},\,
\overrightarrow{v}],
$$$$
\mathcal{P}_3[\overrightarrow{u},\, \overrightarrow{v_1} +
\overrightarrow{v_2}]  = \mathcal{P}_3[\overrightarrow{u},\,
\overrightarrow{v_1}] + \mathcal{P}_3[\overrightarrow{u},\,
\overrightarrow{v_2}],
$$$$ \mathcal{P}_3[\alpha\overrightarrow{u_1},\, \beta\overrightarrow{v_1}]
=\alpha\beta\, \mathcal{P}_3[\overrightarrow{u_1},\,
\overrightarrow{v_1}],
$$$$ \mathcal{P}_3[\overrightarrow{u_1},\, \overrightarrow{u_1}]
=0,
$$ for all %
vectors $\overrightarrow{u_1},\, \overrightarrow{u_2},\,
\overrightarrow{v_1}$ and $\overrightarrow{v_2}$ and numbers
$\alpha,\, \beta\in \mathbb{R}$.

\begin{fact}\label{fact-23}
Let $s(t)$ be a vector equation of the curve. Then this curve can
be locally considered as a circle with radius $$ R =
\frac{|s'(t_0)|^3}{|\mathcal{P}_3[s'(t_0),\, s''(t_0) ]|}.
$$
\end{fact}

\begin{proof}
Let $s(t)$ be an equation of out curve. Consider an arbitrary
moment $t_0$ and $\Delta t \approx 0$. The circle, which we are
talking about passes through points $s(t_0)$, $s(t_0+\Delta t)$
and $s(t_0+2\Delta t)$. Notice that these three are vectors,
whence can be considered as points on a plain, is we assume that
the vectors have start point at the Origin.

We will find now the radius of the circle, which is described
about the triangle with vertices $s(t_0)$, $s(t_0+\Dif t)$ and
$s(t_0+2\Dif t)$.

We can use differentials to express $s(t_0+\Dif t)$ and
$s(t_0+2\Dif t)$ in terms of $s(t_0)$, $s'(t_0)$ and $s''(t_0)$.
Since $$ s'(t_0) =\frac{s(t_0+\Dif t) - s(t_0)}{\Dif t},
$$ then $$ s(t_0 +\Dif t) = s(t_0) + s'(t_0)\Dif t.
$$ Also by $$ s''(t_0) = \frac{s'(t_0+\Dif t) -
s'(t_0)}{\Dif t} = \frac{\frac{s(t_0+2\Dif t) - s(t_0+\Dif
t)}{\Dif t} - s'(t_0)}{\Dif t}
$$ obtain $$
s(t_0 +2\Dif t) = (s''(t_0)\Dif t +s'(t_0))\Dif t + s(t_0 +\Dif
t),
$$ whence $$ s(t_0 + 2\Dif t) = %
s(t_0) +2s'(t_0)\Dif t +s''(t_0)(\Dif t)^2.$$

Thus, we want to find the radius of the circle, which is described
about the triangle with sides $\overrightarrow{a} = s(t_0+\Dif t)
- s(t_0)$, $\overrightarrow{b} = s(t_0+2\Dif t) -s(t_0+\Dif t)$
and $\overrightarrow{c} = s(t_0+2\Dif t) -s(t_0)$. We can express
$\overrightarrow{a}$, $\overrightarrow{b}$ and
$\overrightarrow{c}$ as $$ \overrightarrow{a} = s'(t_0)\Dif t;
$$$$
\overrightarrow{b} = s''(t_0)(\Dif t)^2 + s'(t_0)\Dif t;
$$$$
\overrightarrow{c} = 2s'(t_0)\Dif t +s''(t_0)(\Dif t)^2.
$$

We know, that the radius of the circle, which is descried over the
triangle, can be found as $$ R = \frac{|\overrightarrow{a}|\,
|\overrightarrow{b}|\, |\overrightarrow{c}|}{4A},
$$ where %
$|\overrightarrow{a}|,\, |\overrightarrow{b}|,\,
|\overrightarrow{c}|$ are sides of the triangle and $A$ is its
area.

By Fact~\ref{fact-18}  the area of the triangle, formed by these
vectors is $$ A = \mathcal{P}_3[\overrightarrow{a},\,
\overrightarrow{b}] = \mathcal{P}_3[s'(t_0)\Dif t,\, s''(t_0)(\Dif
t)^2 + s'(t_0)\Dif t]
$$ and can be simplifies as $$ A =\mathcal{P}_3[s'(t_0)\Dif t,\, s''(t_0)(\Dif
t)^2 + s'(t_0)\Dif t] = \mathcal{P}_3[s'(t_0)\Dif t,\,
s''(t_0)(\Dif t)^2 ] =
$$$$
=\mathcal{P}_3[s'(t_0),\, s''(t_0) ](\Dif t)^3
$$
Since $\Dif t\approx 0$, then $(\Dif t)^2$ is much smaller than
$\Dif t$. Whence, we can suppose that $$ |\overrightarrow{a}|
\approx |s'(t_0)|\Dif t,
$$$$
|\overrightarrow{b}| \approx |s'(t_0)|\Dif t
$$ and $$ |\overrightarrow{c}| \approx 2|s'(t_0)|\Dif t.
$$ Thus, $$
R = \frac{|\overrightarrow{a}|\, |\overrightarrow{b}|\,
|\overrightarrow{c}|}{2 \mathcal{P}_3[\overrightarrow{a},\,
\overrightarrow{b} ]} \approx
\frac{|s'(t_0)|^3}{|\mathcal{P}_3[s'(t_0),\, s''(t_0) ]|}.
$$

Notice, that $(\Dif t)^2$ is as close to $0$, comparing with $\Dif
t$, as $\Dif t$ itself is close to $0$. But $\Dif t$ is ``as close
to 0, as possible'', whence $$ R =
\frac{|s'(t_0)|^3}{|\mathcal{P}_3[s'(t_0),\, s''(t_0) ]|}
$$ and the equality is exact, but not just ``approximate''.
\end{proof}

\newpage
\section{Additional mathematical notions and facts}

In this section we will formulate mathematical facts, which we
will use them in our explanations about Kepler's laws. We will try
to explain these facts as clear as possible (without proves, just
explanations of facts themselves), but we will not stay on
technical calculations.

Reader, for whom these facts are new, has to believe them for the
first reading, but to try to understand, what these facts claim,
i.e. to understand, what is stated in the facts, because they will
be used in further computations. The experienced mathematician can
understand the level of math, which will be used later. More or
less detailed explanations of the formulated facts are given in
section~\ref{sect-proves}.

\subsection{Polar coordinates}

Any coordinates of a point in a plane are exactly the way of the
coding the its position by numbers. Precisely, the known cartesian
coordinates of a point $A$ are projections $(x_a,\, y_a)$ of the
vector $OA$ to the fixed ``coordinate lines'' $Ox$ and $Oy$,
passing through the point $O$.

Polar coordinates are constructed as follows. Suppose that we
already have cartesian coordinates (the the origin $O$ and axes
$Ox$ and $Oy$) on the plane. Call the ray $Ox$ the axis (it will
be unique, so we do not need to mention that it is $x$-axis) and
call $\theta$ the angle between $OA$ and $Ox$, calculating from
$Ox$ to $OA$ (for example, points on the positive $y$-axis have
$\theta = 90^0$; points on the negative $y$-axis have $\theta =
270^0$ and points on the line $y=x$ have $\theta = 45^0$ and
$\theta = 135^0$, dependently on the quadrant). Denote $r$ the
length of the segment $OA$, whence the pair
\begin{equation}\label{eq-02} (r,\, \theta),\, r\geq 0,\, \theta \in
[0,\, 360^0) \end{equation} defines all the points on the plain.
Moreover, if $r>0$, then the correspondence is one-to-one and $O$
corresponds to $r=0$ with any $\theta$.

\begin{fact}\label{fact-02}
If $(r,\, \theta)$ are polar coordinates~\eqref{eq-02} of a point
$A$ on a plane, then its cartesian coordinates are
\begin{equation}\label{eq-31}
\left\{\begin{array}{l}x = r\cos\theta,\\
y = r\sin\theta.\end{array} \right.
\end{equation}
\end{fact}

\begin{fact}\label{fact-03}
Suppose that the position-vector of a point is given by polar
coordinates~\eqref{eq-02}. Then speed of this point can be found
by
\begin{equation}\label{eq-12}v^2 = \left( \frac{\Dif r}{\Dif t}\right)^2 +r^2\left(
\frac{d\theta}{dt}\right)^2
\end{equation}
\end{fact}

\subsection{Equation of the plane}

\begin{fact}\label{fact-15}
The set of points on the space, which satisfy the equation $$ Ax+
By  +Cz = D
$$ is a plain. Moreover, the equation of each plain is of this
form.
\end{fact}

\subsection{Ellipse}\label{sect-05}

An ellipse is a plane geometrical figure, which is defined as
follows. Fix two points, say $F_1,\, F_2$ on the plain, which will
contain the ellipse. The ellipse is consisted of all the points
$M$ of the plain such that $MF_1 + MF_2 = 2a$, where $a$ is fixed
at the very beginning. Points $F_1$ and $F_2$ are called focuses
of the ellipse.

Figure~\ref{fig-01} contains an ellipse, whose so called ''focal
line'' (i.e. line, which contains focuses) is the $x$-axis and
origin $O$ is the midpoint of $F_1F_2$. Points of the intersection
of the ellipse and the $x$-axis are denoted $A$ and $C$.
\begin{figure}[h]
\center{\includegraphics[width=8cm]{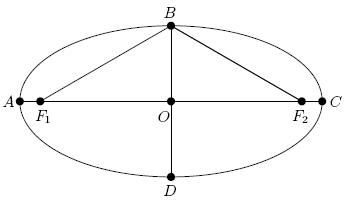}}
\caption{}\label{fig-01}
\end{figure}
It is clear that in this case $a = AO = OC$. Denote $B,\, D$ the
intersections of the ellipse with the $y$-axis and $b = BO = OD$.
Notice, that $F_1O = \sqrt{a^2 -b^2}$, whence $a\geq b$.

\begin{fact}\label{fact-20}
If the Cartesian coordinates of the focuses of the ellipse are
$F_1(-f,\, 0)$ and $F_2(f,\, 0)$, then $$ \frac{x^2}{a} +
\frac{y^2}{b} = 1
$$ is an equation of the ellipse.
\end{fact}

The number $$ e = \frac{\sqrt{a^2 -b^2}}{a} $$ is called
eccentricity of the ellipse. It is clear that $e \in [0,\, 1) $
and $e=0$ if and only if the ellipse is a circle.

\begin{fact}\label{fact-21}
If $F_1$ is in the polar origin and $F_2$ is on the polar axis,
then the equation of the ellipse is $$ r =
\frac{p}{1-e\cos\theta}, $$ where $$ p = \frac{b^2}{a}$$ and is
called semi-latus rectum of the ellipse.
\end{fact}

\begin{fact}\label{fact-22}
The area of the ellipse can be calculated as $A = \pi\, ab.$
\end{fact}

\subsection{Mathematical expression of the Second Kepler's law}

\begin{fact}\label{fact-19} Let the trajectory of the planet is
given in Cartesian coordinates as $$ \frac{x^2}{a^2} +
\frac{y^2}{b^2} = 1,
$$ where $a>b>0$. Then the Second Kepler's law can be rewritten as
$$
x\frac{dy}{dt} -y\frac{dx}{dt} = C,
$$ %
where $C$ is a constant, dependent on the planet and independent
on time. Moreover, $$C = \frac{A}{T}, $$ where $A$ is the area of
the ellipse, which is the trajectory of the planet and $T$ is the
period of the rotation.
\end{fact}

\begin{fact}\label{fact-04}
If the Sun is in the Origin, then the Second Kepler's law in polar
coordinates is expressed as
$$
r^2\frac{d\theta}{dt}=C,
$$ where $C$ is the constant from Fact~\ref{fact-19}.
\end{fact}

\begin{fact}\label{fact-17}
Denote $u = \frac{1}{r}$. Then in polar coordinates the Second
Kepler's Law can be written as $$ v^2 = C^2\left[\left(\frac{\Dif
u}{\Dif t}\right)^2 +u^2\right].
$$
\end{fact}

\newpage
\section{Results about Kepler's laws}

The main deal of this section is to show the equivalence of the
Newton's Laws and the Kepler's Laws. This demonstration will have
the pure nature os some pure mathematical computations. The ideas of these computations is taken from~\cite{Astr}.

Nevertheless, since the Newton's Laws contain the notion of the
force, we need some remark, how we will understand the Force from
the mathematical point of view, since there is not this notion in
mathematics.

We will consider the second Newton's Law as the definition of the
force. In other words, if the mass of some body is fixed (all
masses will be considered as fixed) and force $F$ acts on his
body, then acceleration $$ a = \frac{F}{m}
$$ will appear. %
We will understand any acceleration (calculated mathematically) as
the result of some force and controversially, is the mathematical
calculation will show us am acceleration, the we will interpret it
as the result of some force.

\subsection{Two-bodies problem}\label{sect-07}

Let two bodies with masses $m$ and $M$ (say Sun and a planet) move
such that the 4-th Newton's Law holds. In other words, denote
$r(t)$ the distance between these bodies and assume that there is
a force $$ F(t) = k\cdot \frac{mM}{r^2(t)},
$$ which acts to each of these body in the direction of another.
Clearly, this force (these forces) cause the acceleration
$\overrightarrow{a_S}$ of the Sun with absolute value $$
|\overrightarrow{a_S}| = k\cdot \frac{m}{r^2(t)}
$$
and the acceleration of the planet $\overrightarrow{a_P}$ with
absolute value
$$
|\overrightarrow{a_P}| = k\cdot \frac{M}{r^2(t)}.
$$

We will fix the Sun at the origin. This will mean that the speed
and the acceleration of the Sun will be $0$. It will give all
other points of the coordinate system the additional acceleration
$a_s$, directed to the Origin. Thus, the acceleration
$\overrightarrow{a}$ of the planet in ``the fixed'' coordinate
system will be $$ |\overrightarrow{a}| = |\overrightarrow{a_S}
-\overrightarrow{a_P}|.
$$ Since the directions of
$\overrightarrow{a_S}$ and $\overrightarrow{a_P}$ are opposite,
then $$ |a(t)| = k\cdot \frac{m}{r^2(t)} + k\cdot \frac{M}{r^2(t)}
= \frac{k(m+M)}{r^2(t)}.
$$ %
Denote $\mu = k(m+M)$ and we are ready now to formulate the
mathematical problem, which we will solve.\\

\textbf{\underline{Mathematical problem}}. \emph{Denote}
$$\overrightarrow{s}(t) = (x(t),\, y(t),\, z(t))$$ \emph{the position
vector of some point and denote} $$r(t) = \sqrt{x^2(t) + y^2(t)
+z^2(t)}.$$ \emph{Suppose that} %
\begin{equation}\label{eq-5-14} %
\frac{\Dif^2\overrightarrow{s}}{\Dif t} = -\,
\frac{\mu}{r^2(t)}\cdot
\frac{\overrightarrow{s}(t)}{|\overrightarrow{s}(t)|}.
\end{equation} \emph{We need to find (to describe) the function %
$\overrightarrow{s}(t)$. Precisely, we need to prove that
$\overrightarrow{s}(t)$ satisfies the Kepler's Laws}.\\

Notice, that the equation~\eqref{eq-5-14} means nothing more than
the acceleration $\overrightarrow{a}
=\frac{\Dif^2\overrightarrow{s}}{\Dif t}$ has the absolute value
$\frac{\mu}{r^2(t)}$ (we suppose $\mu$ to be positive) and the
multiplier $\frac{\overrightarrow{s}(t)}{|\overrightarrow{s}(t)|}$
together with minus-sign before the entire expression mans that
the direction of $\overrightarrow{a}$ is opposite to
$\overrightarrow{s}$, i.e. the acceleration is directed from the
position of the Planet to the origin.

We can rewrite~\eqref{eq-5-14} as
\begin{equation}\label{eq-5-15}
\begin{array}{lllllll}
\displaystyle{\frac{d^2x}{dt^2} = -\mu\, \frac{x}{r^3}}, & &
\displaystyle{\frac{d^2y}{dt^2} = -\mu\, \frac{y}{r^3}}, & &
\displaystyle{\frac{d^2z}{dt^2} = -\mu\, \frac{z}{r^3}}.
\end{array}
\end{equation}

The obtained~\eqref{eq-5-15} is called the ``second order system
of differential equations''. The name ``differential equation'' is
because the differentials (derivatives) are parts of equations of
this system. The order of a system of differential equations is
the maximal derivative, which appears in the equations (clearly,
this maximum here is 2, being, for instance, the order of the
derivative $\frac{d^2x}{dt^2}$, but the third or more order
derivative does not appear in the system of equations).

Notice, that tree equations of~\eqref{eq-5-15} are not
independent, because $r$, being the distance from the planet to
the origin, is dependent on $x,\, y,\, z$, precisely, for example
the first equation of~\eqref{eq-5-15} is, in fact, $$
\frac{d^2x}{dt^2} = \frac{-\mu\, x}{(\sqrt{x^2 + y^2 +z^2})^3}
$$ and in the same way other two equations can be rewritten.

There are no general rules of solving the systems of differential
equations. Students of mathematical faculties of universities
study the classical course, which is called differential equations
theory. They study there the methods of solving different types of
differential equations, or system of differential equations.
Clearly, specialists on differential equations consider the
equations, which are studied in the university course to by
``standard'' and ``the simplest''. Knowing of methods of solving
of ``standard differential equations'' is considered like a ``time
table'', which should be clear for everybody, who studies
differential equations. If a differential equation has no standard
method to solve, the it is necessary to combine the knowing
standard methods, or in any way use the talent or intuition.

The useful notion, which can help to solve a differential equation
is so called \textbf{the first integral}. Let one have a system of
differential equations with respect to unknown functions $x(t),\,
y(t)$ and $z(t)$ (as in our case). The first integral if some
functional expression, which contains the symbols for $t$, $x$,
$y$, $z$ and all the derivatives, such that the substitution of
the solution to this expression transforms it to zero. The
usefulness of the first integral is in that, it may help to make
some conclusions about the properties of the solution, without
finding the solutions explicitly. Clearly, more than one the first
integral can be necessary to find the final (explicit) solution of
the differential equation.

Let us start to find on of the first integral of~\eqref{eq-5-15}.
Multiply the second and third equation of~\eqref{eq-5-15} by $z$
and $y$ respectively and obtain $$
\begin{array}{lllll}
\displaystyle{z\, \frac{d^2y}{dt^2} = -\mu\, \frac{yz}{r^3},}&&
\displaystyle{y\, \frac{d^2z}{dt^2} = -\mu\, \frac{yz}{r^3}.}
\end{array}
$$
The substraction of these equations leads to
\begin{equation}\label{eq-15}y\, \frac{d^2z}{dt^2} - z\,
\frac{d^2y}{dt^2} =0.\end{equation}

Notice that
\begin{equation}\label{eq-16}\frac{d}{dt}\left( y\, \frac{dz}{dt} - z\,
\frac{dy}{dt}\right) = y\, \frac{d^2z}{dt^2} - z\,
\frac{d^2y}{dt^2},
\end{equation} because $$
\frac{d}{dt}\left( y\, \frac{dz}{dt} - z\, \frac{dy}{dt}\right) =
\left(\underline{\frac{dy}{dt}\, \frac{dz}{dt}} + y\,
\frac{d^2z}{dt^2}\right) -\left(\underline{\frac{dz}{dt}\,
\frac{dy}{dt}} + z\, \frac{d^2y}{dt^2}\right) = $$ $$=y\,
\frac{d^2z}{dt^2} - z\, \frac{d^2y}{dt^2}.
$$

Denote $ q(t) = y\, \frac{dz}{dt} - z\, \frac{dy}{dt},$ whence %
rewrite equalities~\eqref{eq-15} and~\eqref{eq-16} as $
\frac{dq}{dt} =0. $ This means that $q$ is a constant function,
i.e. there exists $A = \const$ such that $$ y\, \frac{dz}{dt} -
z\, \frac{dy}{dt} = A.
$$

Do the analogical transformations with another pairs od equations
of~\eqref{eq-5-15}. Multiply the first and the third equation by
$z$ and $x$, then subtract them and obtain $\frac{d}{dt}\left( z\,
\frac{dx}{dt} - x\, \frac{dz}{dt}\right)=0$. Similarly obtain
$\frac{d}{dt}\left( x\, \frac{dy}{dt} -y\, \frac{dx}{dt}\right)
=0$ from the first two equations of~\eqref{eq-5-15}.
Thus, %
\begin{equation}\label{eq-5-15a} \left\{
\begin{array}{l}
y\, \frac{dz}{dt} - z\, \frac{dy}{dt} = A,\\
z\, \frac{dx}{dt} - x\, \frac{dz}{dt} = B,\\
x\, \frac{dy}{dt} - y\, \frac{dx}{dt} = C,
\end{array}\right.
\end{equation} where $A,\, B,\, C$ are some constants.

Multiply equations of~\eqref{eq-5-15a} by $x$, $y$ and $z$
respectively, then add them and obtain \begin{equation}
\label{eq-17}Ax + By +Cz =0.
\end{equation} By Fact~\ref{fact-15}, This equation describes a
plain, i.e. the orbit of the planet is a plan curve and the Sun
(the Origin) belongs to the plain of this curve.

Notice ones more, that~\eqref{eq-17} is one more example of the
first integral of~\eqref{eq-5-15}.

Since the movement of the planet is plain, then we can simplify
the former system of equations~\eqref{eq-5-15}. Suppose that the
orbit belongs to the $XOY$-plain, whence $z(t)$ will become a
zero-function, and~\eqref{eq-5-15a} will be simplified as
$$ x\, \frac{dy}{dt} -y\, \frac{dx}{dt} = C_1,
$$ where $C_1$ is %
sone new constant, which is not necessary to be equal to $C$
from~\eqref{eq-5-15a}.

Notice, that the obtained equality is exactly the
\textbf{\underline{Second Kepler's Law}} from Fact~\ref{fact-19}.

We will find now one more first integral of~(\ref{eq-5-15}).
Multiply the first equality by $2\, \frac{dx}{dt}$ and the second
by $2\frac{dy}{dt}$. After the addition obtain
\begin{equation}\label{eq-18} %
2\left[ \frac{dx}{dt}\, \frac{d^2x}{dt^2} +\frac{dy}{dt}\,
\frac{dy}{dt^2}\right] = -\, \frac{2\mu}{r^3}\left[ x\frac{dx}{dt}
+y\, \frac{dy}{dt}\right].
\end{equation} Remind, that we have already obtained got that
$z=0$, thus~\eqref{eq-5-15} contains only two equations.

For the simplification of~\eqref{eq-18}, observe that
$$ \frac{\Dif
v^2}{\Dif t} = 2\frac{dx}{dt}\, \frac{d^2x}{dt^2} +
2\frac{dy}{dt}\, \frac{d^2y}{dt^2}
$$ and
$$
\frac{d}{dt}(r^2) = 2x\, \frac{dx}{dt} + 2y\, \frac{dy}{dt}.
$$

Indeed, $$%
v^2 = \left(\frac{dx}{dt}\right)^2 + \left(\frac{dy}{dt}\right)^2
$$ implies $$
\frac{dv^2}{dt} =
\frac{d}{dt}\left(\left(\frac{dx}{dt}\right)^2\right) +
\frac{d}{dt}\left(\left(\frac{dy}{dt}\right)^2\right) =
$$
$$
=2\frac{dx}{dt}\, \frac{d^2x}{dt^2} + 2\frac{dy}{dt}\,
\frac{d^2y}{dt^2}.
$$
From another hand, since
$$ r^2 =
x^2 + y^2,
$$ then $$
\frac{d}{dt}(r^2) = \frac{d}{dt}\left( x^2 +y^2 \right) = 2x\,
\frac{dx}{dt} + 2y\, \frac{dy}{dt}.
$$ Thus, the equation~\eqref{eq-18} is equivalent to
\begin{equation}\label{eq-13}
\frac{\Dif v^2}{\Dif t} = -\, \frac{2\mu}{r^3}\, \frac{\Dif
r^2}{\Dif t}.
\end{equation}

Moreover, notice that
$$-\frac{\mu}{r^3}\frac{d}{dt}(r^2) =
-\frac{\mu}{r^3}\cdot r\, \frac{dr}{dt} = -\frac{\mu}{r^2}\,
\frac{dr}{dt}$$ and
$$
\frac{d}{dt}\left(\frac{2\mu}{r}\right)= -\frac{\mu}{r^2}\,
\frac{dr}{dt}.
$$ Thus, we can rewrite~\eqref{eq-13} as
$$\frac{\Dif v^2}{\Dif t} = \frac{d}{dt}\left(\frac{2\mu}{r}\right),
$$ or
$$\frac{d}{dt}\left( v^2 -\frac{2\mu}{r}\right) =0. $$
Thus, there is a constant $h$ such that
\begin{equation}\label{eq-5-17} v^2 -\frac{2\mu}{r} =h.\end{equation}

We are ready now to obtain the equation of the movement of the
planet. We will make the further calculations in polar
coordinates.

By Fact~\ref{fact-17}, we can rewrite the Second Kepler's Law in
polar coordinates as
$$
v^2 = c^2\left[\left(\frac{\Dif u}{\Dif t}\right)^2 +u^2\right],
$$ where $c$ is some constant and $u = \frac{1}{r}$. This lets us
to rewrite~\eqref{eq-5-17} as
$$ c^2\left(
\frac{du}{d\theta}\right)^2 = h + 2\mu u - c^2u^2.
$$
This differential equation can be reduced to one of separate
variables as
\begin{equation}\label{eq-5-19}
d\theta = \pm\, \frac{d(cu)}{\sqrt{h + 2\mu u -c^2u^2}}.
\end{equation}
Keeping in mid the fact that $$ \int\frac{-\Dif x}{\sqrt{1-x^2}} =
\arccos x,
$$ transform the expression under the square root in the
denominator of~\eqref{eq-5-19} as $$ h + 2\mu u -c^2u^2 = \left(
\frac{\mu^2}{c^2} +h \right) - \left( cu -\frac{\mu}{c}\right)^2.
$$ Denote $q = \frac{\mu^2}{c^2} +h$ and $\xi =cu -\frac{\mu}{c}$,
notice that $\Dif \left(cu - \frac{mu}{c}\right) = \Dif \xi$,
whence rewrite~\eqref{eq-5-19} as $$ d\theta = \pm\,
\frac{d\xi}{\sqrt{q^2 -\xi^2}} = \pm\, \frac{d\left(
\frac{\xi}{q}\right)}{\sqrt{1-\left( \frac{\xi}{q}\right)^2}}.
$$ Then $$\pm\, \theta =
\arccos\left( \frac{\xi}{q}\right) +k.
$$ Whence $$
\frac{\xi}{q} = \pm\cos(\theta -k),
$$ or $$
\xi = \pm\,q\cos(\theta-k).
$$ Plug the previous expressions for $q$, $\xi$ and %
$u$ and obtain
$$
\frac{c}{r} - \frac{\mu}{c} = \pm\, \sqrt{\frac{\mu^2}{c^2}+h}\,
\cos(\theta-k),
$$ or $$
\frac{c}{r} = \frac{\mu}{c}\left[ 1 \pm\,
\sqrt{1+\frac{c^2h}{\mu^2}}\, \cos(\theta-k)\right].
$$ Resolve the obtained %
expression with respect to $r$ and obtain
\begin{equation}\label{eq-5-22}
r = \frac{\frac{c^2}{\mu}}{1 \pm\sqrt{1+\frac{c^2h}{\mu^2}}\,
\cos(\theta-k)}.
\end{equation}
Notice, that choose $\theta$ for $\theta+180^0$ means the choose
of the direction of the polar axis, whence without loos of
generality assume that we have ``--'' instead of $\pm$. Now denote
$p=\frac{c^2}{\mu}$, $e = \sqrt{1+\frac{c^2h}{\mu^2}}$ and $\nu
=\theta -k$ for obtain $$ r = \frac{p}{1- e\cos\nu},
$$ which is the known equation of the curve of the second order.

Since all the closed curves of the second order are ellipse,
hyperbola and parabola, claim that for the planets from the Solar
system trajectory is an ellipse. Clearly, that for other pair of
bodies (for example, the Sun and some comets) this trajectory can
be as hyperbola, as parabola too.

\subsection{Equivalence of the Third Kerpler's Law and the fourth
Newton's law}

\begin{fact}\label{fact-08}
Suppose that the Sun is fixed in the Origin and a planet moves by
the elliptic orbit with polar equation $$r = \frac{p}{1-e\,
\cos\theta}$$ and the equation
\begin{equation}\label{eq-56}
r^2\, \frac{\Dif \theta}{\Dif t} = C \end{equation} holds. Then
the acceleration of the planet is
\begin{equation}\label{eq-53}
\overrightarrow{a} = -\,
\frac{\overrightarrow{r}}{|\overrightarrow{r}|}\cdot
\frac{C^2}{p\, r^2}. \end{equation}
\end{fact}

The proof of this fact is simple, but technical. The idea of the
proof is the following. Since we have polar coordinates, we have
equalities \begin{equation}\label{eq-57}\left\{ \begin{array}{l}x = r\cos \theta;\\
y = r\sin\theta,\end{array}\right. \end{equation} where both
$\theta$ and $r$ are unknown functions, dependent on $t$.
Formula~\eqref{eq-56} lets us to express $x$ and $y$ in terms of
$\theta$ as $$ x(t) = r(t)\cos\theta(t) = \frac{p\,
\cos\theta(t)}{1-e\sin\theta(t)}
$$ and analogously for $y$.

By Fact~\ref{fact-09} we need to find the second derivative of the
position function $s(t) = (x(t),\, y(t))$. We will use
Fact~\ref{fact-01}, which means that we will find separately the
derivatives $x''(t)$ and $y''(t)$, saying that $a(t) = (x''(t),\,
y''(t))$.

We can find the derivative $\theta'(t)$ from~\eqref{eq-56},
because, in fact, this equation is nothing, but this derivative.
We will use this derivative together with the Chain rule
(Fact~\ref{fact-10}), Product rule (Fact~\ref{fact-11}) and some
table derivatives (Fact~\ref{fact-12}).

The proof of Fact~\ref{fact-08} in details is given at
Section~\ref{sect-04}.

Fact~\ref{fact-08} has an important corollary. The proof of this
corollary uses only Fact~\ref{fact-08} and some technical
calculations with parameters of an ellipse, which were introduced
in Section~\ref{sect-05}.

\begin{fact}\label{fact-13}
Suppose that the planets moves such that Newton's the First and
the Second Kepler's laws hold. Then the Third Kepler's law is
equivalent to the existence of the constant $k$ (which is the same
for all planets) such that
\begin{equation}\label{eq-55}|\overrightarrow{a}| =
\frac{k}{r^2},\end{equation} where $a$ is the acceleration of a
planet and $r$ is its distance to the Sun.
\end{fact}

\begin{proof}
Suppose that the Third Kepler's law holds, i.e. there exists a
constant $d$ (which is the same for all planets) such that
\begin{equation}\label{eq-51}d = \frac{a^3}{T^2},\end{equation}
where $a$ is the half of the bigger diameter of the ellipse, which
is the trajectory of the planet (we will use notations from
Section~\ref{sect-05}). By Facts~\ref{fact-21} and~\ref{fact-22},
i.e. using formulas $$A = \pi\, ab $$ and $$ p = \frac{b^2}{a},$$
we can rewrite~\eqref{eq-51} as
$$ \pi^2\, d = \frac{\pi^2 a^3}{T^2} = \frac{\pi^2
a^2b^2}{T^2}\cdot \frac{a}{b^2} = \frac{A^2}{p\, T^2}.$$ Next, by
Fact~\ref{fact-19} we have $A = CT$ (where $C$ is taken from the
Second Kepler's Law, i.e. it is the same $C$, which is in the
formulation of Fact~\ref{fact-08}). Thus, using
Fact~\ref{fact-08}, calculate the acceleration~\eqref{eq-53} as $$
|\overrightarrow{a}| = \frac{C^2}{p\, r^2} = \frac{A^2}{p\,
r^2T^2} = \frac{\pi^2\, d}{r^2}.
$$ Now denote $k=\pi^2\, d$ and~\eqref{eq-55} follows.

Conversely, suppose~\eqref{eq-55}. By Fact~\ref{fact-08} write $$
|\overrightarrow{a}|= \frac{k}{r^2} = \frac{C^2}{p\, r^2}.
$$ Using Fact~\ref{fact-22} %
rewrite the obtained equality as $$
k = \frac{A^2}{p\, C^2}
$$
Express, as in the fist part of the proof, $A$ and $p$ by
Facts~\ref{fact-21} and~\ref{fact-22} obtain $$ k = \frac{\pi^2\,
a^2b^2}{C^2}\cdot \frac{a}{b^2} = \frac{\pi^2\, a^3}{T^2}.
$$ Denote $$d = \frac{k}{\pi^2}$$ and we are done.
\end{proof}

\newpage

\section{Analytic representation of the movement}

Suppose that we know, that a planet moves according to Kepler's
Laws. In this section we will use mathematical tools to specify
some properties of the movement of this planet.

Since we know that the trajectory of the planet is an ellipse, we
can write its polar equation as \begin{equation}\label{eq-59}r(t)
= \frac{p}{1-\varepsilon\cos \theta(t)}.
\end{equation}
Thus the function $\theta = \theta(t)$ define the law of the
movement of our planet by its orbit. This function can be found
from the Second Kepler's Law, written as in Fact~\ref{fact-04},
i.e.
$$ r^2\, \frac{\Dif \theta}{\Dif t} =
C.
$$ We can use this and the above equation of the ellipse to
express the derivative of $\theta$ as
\begin{equation}\label{eq-43}
\theta' =\frac{C}{p^2}\, (1-\varepsilon\cos \theta)^2,
\end{equation} whence rewrite the Second Kepler's Law as
\begin{equation}\label{eq-60}%
\frac{\Dif \theta}{(1-\varepsilon\cos \theta)^2}
=\frac{C}{p^2}\, \Dif t.
\end{equation}
This is the differential equation of one of the simplest form,
which is called differential equation with separated variables.
Using quite standard techniques from the mathematical analysis, we
can solve this it and obtain the evident formulas, which will
express $t$ in terms of $\theta$.

Also we can use these formulas (after we will obtain them) for
plotting (and understanding) some properties of the movement of
planets by their orbits.

\subsection{The parametric equations %
of the motion of a planet around the sun %
and their corollaries}\label{sect-06}

We will find the expression for $\theta(t)$ from the
equation~\eqref{eq-60} and this will give the law of the motion of
the planet on its orbit.

Denote
\begin{equation}\label{eq-42} \Theta(\theta) = \frac{1}{(1-\varepsilon\cos
\theta)^2}.\end{equation}

According to Fact~\ref{fact-16}, we can write the solution
of~\eqref{eq-60} as $$ \int \Theta(\theta)\Dif\theta -
\int\frac{C}{p^2}\Dif t = \mathcal{T}(t),
$$ where $\mathcal{T}'(t) =0$ for all $t$.

Denote $I(\theta)$ an infinite integral of $\Theta(\theta)$, i.e.
$I(\theta)$ is a function such that $I'(\theta) =\Theta(\theta)$.
By Newton-Leibnitz Theorem we can determine $I(\theta)$ as
\begin{equation}\label{eq-63}%
I(\theta) = \int\limits_0^\theta \frac{\Dif x}{(1-\varepsilon \cos
x)^2}. \end{equation} This implies that the solution
of~\eqref{eq-60} is
\begin{equation}\label{eq-68}\int\limits_0^\theta \frac{\Dif
x}{(1-\varepsilon \cos x)^2} = \frac{C}{p^2}\cdot t
+\mathcal{T}(t), \end{equation} where $\mathcal{T}'(t) =0$ for all
$t$.\\

Equation~\eqref{eq-68} can be considered as law, which lets to
find an polar angle $\theta$ of the planet at the time $t$. The
value $\frac{C}{p^2}$ depends on ``how quickly the time passes'',
i.e. whether we measure time in seconds, hours, years etc.

\begin{figure}[p]
\begin{minipage}[h]{0.45\linewidth}
\center{\includegraphics[width=7.5cm]{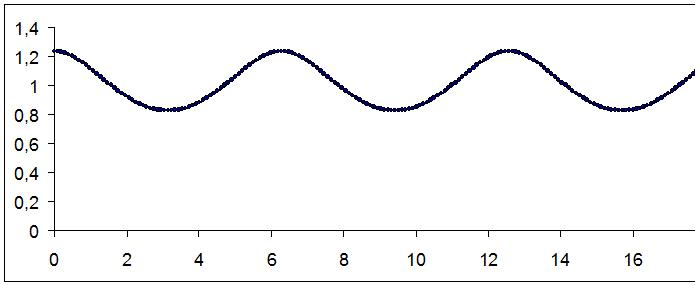}}\\
\centerline{a. $\varepsilon = 0,1$}
\end{minipage}
\begin{minipage}[h]{0.45\linewidth}
\center{\includegraphics[width=7.5cm]{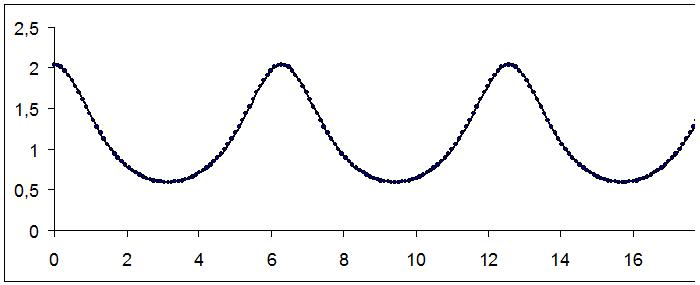}}\\%
\centerline{b. $\varepsilon = 0,3$}%
\end{minipage}
\caption{Graph of %
$\Theta$ of the form~\eqref{eq-42} }\label{fig-16}
\end{figure}

\begin{figure}[p]
\begin{minipage}[h]{0.45\linewidth}
\center{\includegraphics[width=7.5cm]{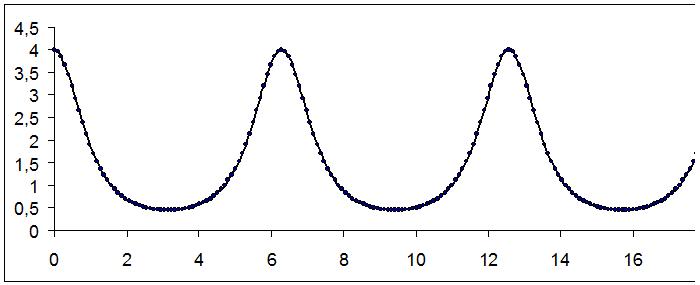}}\\
\centerline{a. $\varepsilon = 0,5$}
\end{minipage}
\begin{minipage}[h]{0.45\linewidth}
\center{\includegraphics[width=7.5cm]{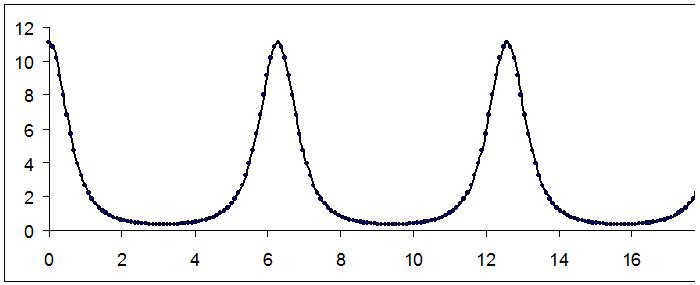}}\\%
\centerline{b. $\varepsilon = 0,7$}%
\end{minipage}
\caption{Graph of %
$\Theta$ of the form~\eqref{eq-42}}\label{fig-17}
\end{figure}

\begin{figure}[p]
\center{\includegraphics[width=7.5cm]{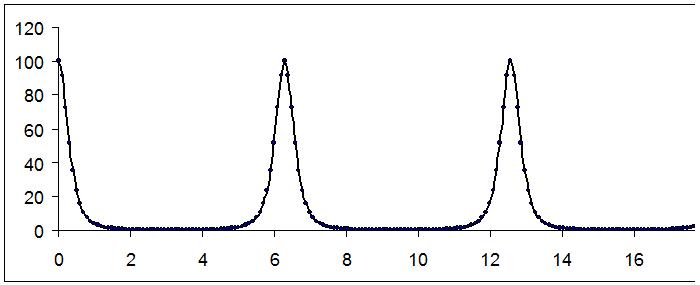}}
\caption{Graph of %
$\Theta$ of the form~\eqref{eq-42} for $\varepsilon =
0,9$}\label{fig-18}
\end{figure}

The graphs of the function $\Theta(\theta)$ for $\theta\in [0,\,
3\pi]$ are given at Figures~\ref{fig-16}, \ref{fig-17}
and~\ref{fig-18}.

The explicit formula for the $I(\theta)$ of the form~\eqref{eq-63}
can be found with the use of the following fact.

\begin{fact}\label{fact-14}
Suppose that $\varepsilon\in (0,\, 1)$. Then $$\int \frac{\Dif
x}{(1-\varepsilon \cos x)^2} = 
\frac{\sqrt{1-\varepsilon}}{\sqrt{1+\varepsilon}}\cdot
\frac{2}{(1-\varepsilon^2)(1-\varepsilon)} \left( \arctan t
 + \frac{t\, \varepsilon}{t^2+1} \right),
$$
where $t = \tan\left( \frac{x}{2}\cdot
\frac{\sqrt{1+\varepsilon}}{\sqrt{1-\varepsilon}}\right).$
\end{fact}

We prove Fact~\ref{fact-14} in Section~\ref{sect-4}.

Using Fact~\ref{fact-14}, we can write the solution~\eqref{eq-68}
of~\eqref{eq-60} as
\begin{equation}\label{eq-41}\frac{C}{p^2}\, t = %
\mathcal{I}(\theta) +\mathcal{T}(t),\end{equation} where
\begin{equation}\label{eq-64}
\begin{array}{rl}
\mathcal{I}(\theta)
=&\frac{\sqrt{1-\varepsilon}}{\sqrt{1+\varepsilon}}\cdot
\frac{2}{(1-\varepsilon^2)(1-\varepsilon)}\, \times \\
&\times\left( \arctan \left( \tan\left( \frac{\theta}{2}\cdot
\frac{\sqrt{1+\varepsilon}}{\sqrt{1-\varepsilon}}\right) \right)
 + \frac{\varepsilon \cdot \tan\left( \frac{\theta}{2}\cdot
\frac{\sqrt{1+\varepsilon}}{\sqrt{1-\varepsilon}}\right)}{\tan^2\left(
\frac{\theta}{2}\cdot
\frac{\sqrt{1+\varepsilon}}{\sqrt{1-\varepsilon}}\right) +1}
\right)\end{array} \end{equation} and $$ \mathcal{T}'(t)=0
$$ for all $t$.

This formula lets to plot (clearly, by some computer techniques,
not manually) the graph of the function $t = t(\theta)$, i.e. plot
the dependence of time and the the polar angle of the planet.
Remind ones more, that the multiplier $\frac{C}{p^2}$ determines
``how quick the time passes'' and an be ignored (can be supposed
to be equal 1) in our computations.

The graph of the equation~\eqref{eq-41} is given at
pictures~\ref{fig-02}b and~\ref{fig-03}b (for $\varepsilon = 0,1$
and $\varepsilon = 0,3$ respectively).

\begin{figure}[h]
\begin{minipage}[h]{0.45\linewidth}
\center{\includegraphics[width=7.5cm]{./1-2.jpg}}\\
\centerline{a. Graph of %
$\Theta$ of the form~\eqref{eq-42}}
\end{minipage}
\begin{minipage}[h]{0.45\linewidth}
\center{\includegraphics[width=7.5cm]{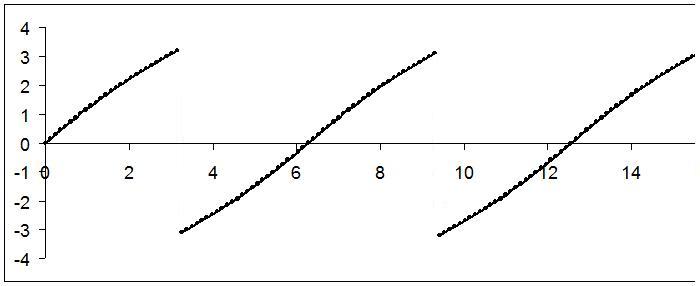}}\\%
\centerline{b. Graph of $I(\theta)$
of the form~\eqref{eq-63}}%
\end{minipage}
\caption{$\varepsilon = 0,1$}\label{fig-02}
\end{figure}

\begin{figure}[h]
\begin{minipage}[h]{0.45\linewidth}
\center{\includegraphics[width=7.5cm]{./3-2.jpg}}\\
\centerline{a. Graph of %
$\Theta$ of the form~\eqref{eq-42}}
\end{minipage}
\begin{minipage}[h]{0.45\linewidth}
\center{\includegraphics[width=7.5cm]{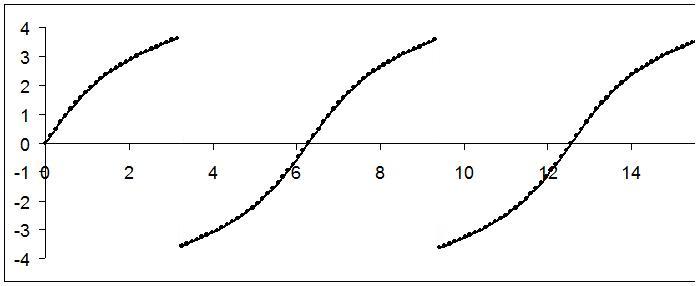}}\\%
\centerline{b. Graph of $I(\theta)$
of the form~\eqref{eq-63}}%
\end{minipage}
\caption{$\varepsilon = 0,3$}\label{fig-03}
\end{figure}

It is clear from the graphs, that the functions, which are plotted
at Figures~\ref{fig-02}b and~\ref{fig-03}b are discontinuous,
whence they can not be areas under the graphs of functions from
Figures~\ref{fig-02}a and~\ref{fig-03}a respectively.

The fact that $\mathcal{T}$ from~\eqref{eq-41} is piecewise
constant, means that we may arbitrary move in vertical direction
the parts of the graphs from Figures~\ref{fig-02}b
and~\ref{fig-03}b, obtaining the function, which also be an
integral of $\Theta$. Since we a looking for the certain area
$I(\theta) = \int\limits_0^\theta \Theta(x)\Dif x$, then we need
the integral $\widetilde{\mathcal{I}}(\theta)$ of $\Theta$ with
the following properties:

1. $\widetilde{\mathcal{I}}(0) = 0$, because we calculate the area
in the interval $[0,\, \theta]$, whence the area from $0$ to $0$
should be $0$.

2. The integral has to be continuous, because the former function
$\Theta$ is continuous, whence the area under is is also
continuous.

The correspond transformation with $\mathcal{I}(\theta)$ of the
form~\eqref{eq-64} for $\varepsilon =0.3$ (i.e. in fact with the
graph from Fig.~\ref{fig-03}b) are given at Fig~\ref{fig-07}.

\begin{figure}[p]
\begin{center}
\center{\includegraphics[width=7.5cm]{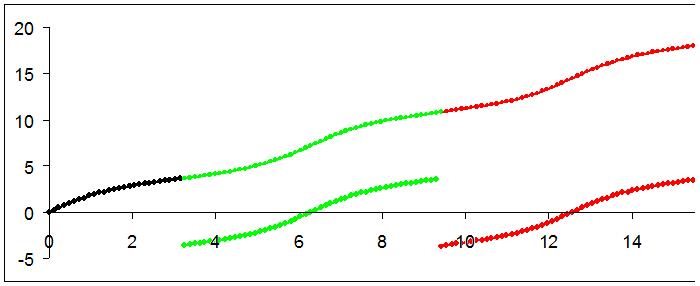}}\\
\end{center}
\caption{Correct graph of $I(\theta)$ of the form~\eqref{eq-63}
for $\varepsilon = 0.3$}\label{fig-07}
\end{figure}

\begin{figure}[p]
\begin{minipage}[h]{0.45\linewidth}
\center{\includegraphics[width=7.5cm]{./5-2.jpg}}\\
\centerline{a. Graph of %
$\Theta$ of the form~\eqref{eq-42}}
\end{minipage}
\begin{minipage}[h]{0.45\linewidth}
\center{\includegraphics[width=7.5cm]{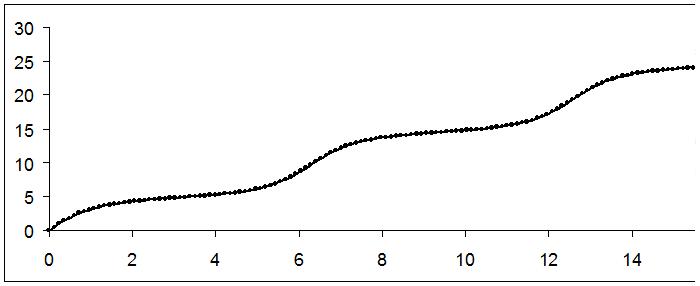}}\\%
\centerline{b. Graph of $I(\theta)$
of the form~\eqref{eq-63}}%
\end{minipage}
\caption{$\varepsilon = 0,5$}\label{fig-04}
\end{figure}

\begin{figure}[p]
\begin{minipage}[h]{0.45\linewidth}
\center{\includegraphics[width=7.5cm]{./7-2.jpg}}\\
\centerline{a. Graph of %
$\Theta$ of the form~\eqref{eq-42}}
\end{minipage}
\begin{minipage}[h]{0.45\linewidth}
\center{\includegraphics[width=7.5cm]{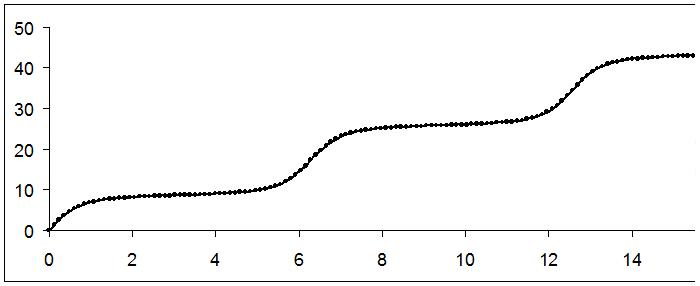}}\\%
\centerline{b. Graph of $I(\theta)$
of the form~\eqref{eq-63}}%
\end{minipage}
\caption{$\varepsilon = 0,7$}\label{fig-05}
\end{figure}

\begin{figure}[p]
\begin{minipage}[h]{0.45\linewidth}
\center{\includegraphics[width=7.5cm]{./9-2.jpg}}\\
\centerline{a. Graph of %
$\Theta$ of the form~\eqref{eq-42}}
\end{minipage}
\begin{minipage}[h]{0.45\linewidth}
\center{\includegraphics[width=7.5cm]{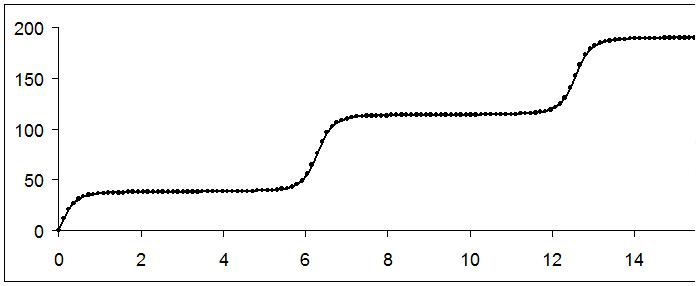}}\\%
\centerline{b. Graph of $I(\theta)$
of the form~\eqref{eq-63}}%
\end{minipage}
\caption{$\varepsilon = 0,9$}\label{fig-06}
\end{figure}

The graphs of $\Theta(x)$ of the form~\eqref{eq-42} and areas
under these graphs, i.e. functions $I(\theta)$ of the
form~\eqref{eq-63} for different $\varepsilon$ are given at
Figs.~\ref{fig-04}, \ref{fig-05} and~\ref{fig-06}.

\subsection{The speed of the movement}

We have seen in Section~\ref{sect-06} that the law of the
dependence of $\theta$ on $t$ (especially in the case when
$\varepsilon \approx 1$) can be complicated. In the same way we
can establish the dependence of the speed of the planet on time.
Using Fact~\ref{fact-03} write $$v^2 = \left(\frac{\Dif r}{\Dif
t}\right)^2 +r^2 \left(\frac{\Dif \theta}{\Dif t}\right)^2.$$ If
follows from~\eqref{eq-59} by chain rule that
$$ r'(t)=
\frac{-p\, \varepsilon\sin\theta}{(1-\varepsilon\cos\theta)^2}\,
\theta'(t).
$$ Thus, using~\eqref{eq-43}, obtain
$$v^2 = \left( %
\left(\frac{-p\, \varepsilon\sin\theta}%
{(1-\varepsilon\cos\theta)^2}\right)^2 + r^2\right)\, \theta'^2 =
$$$$
= \left( %
\left(\frac{-p\, \varepsilon\sin\theta}%
{(1-\varepsilon\cos\theta)^2}\right)^2 + \left(
\frac{p}{1-\varepsilon\cos \theta(t)}\right)^2\right) \cdot
\frac{C^2}{p^4}\, (1-\varepsilon\cos \theta)^4 =
$$$$
= \left(\left(\frac{\varepsilon\sin\theta}%
{1-\varepsilon\cos\theta}\right)^2 + 1 \right)\cdot
\frac{C^2}{p^2}\cdot (1-\varepsilon\cos\theta)^2.
$$

Graphs of the speed in the assumption $\frac{C}{p}=1$, dependent
on angle, i.e. graphs of the equations $$
|\overrightarrow{v}| = \sqrt{\left(\left(\frac{\varepsilon\sin\theta}%
{1-\varepsilon\cos\theta}\right)^2 + 1 \right)\cdot
(1-\varepsilon\cos\theta)^2}
$$
are given on Figures~\ref{fig-08}, \ref{fig-09} and~\ref{fig-10}
for $\theta\in [0,\, 2\pi]$. These graphs show, how many times the
biggest speed is greater the smallest, dependently on
$\varepsilon$.

Notice, that we have obtained in Section~\ref{sect-07} the
formula~\eqref{eq-5-17}, i.e. $$ v^2 -\frac{2\mu}{r} = h,
$$ where $\mu$ and $h$ are constants. This formula lets us to make
the conclusion, that the speed of the planet strongly depends on
$r$ and is as huge, as small is $r$. The constructed graphs
confirm this claim. Indeed, the smallest speed appear to be at
$\theta = 0$ and the biggest at $\theta = \pi$.

\begin{figure}[p]
\begin{minipage}[h]{0.45\linewidth}
\center{\includegraphics[width=7.5cm]{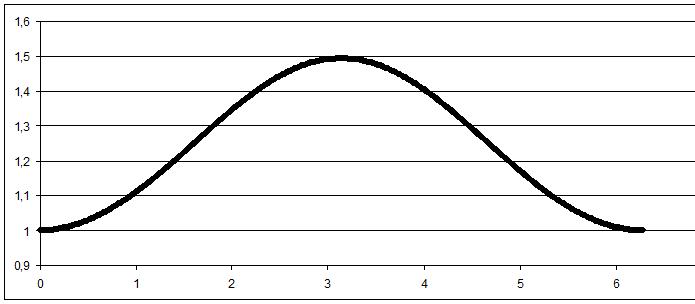}}\\
\centerline{a. $\varepsilon = 0,1$}
\end{minipage}
\begin{minipage}[h]{0.45\linewidth}
\center{\includegraphics[width=7.5cm]{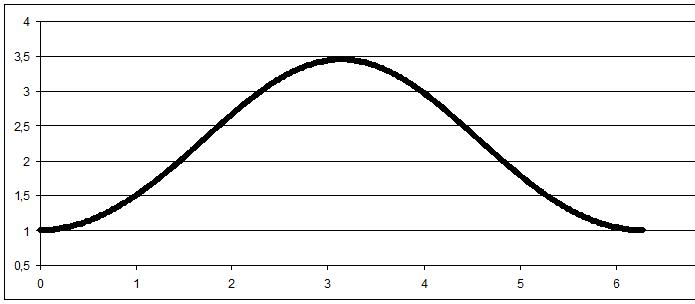}}\\%
\centerline{b. $\varepsilon = 0,3$}%
\end{minipage}
\caption{Graphs of the speed of the planet}\label{fig-08}
\end{figure}

\begin{figure}[p]
\begin{minipage}[h]{0.45\linewidth}
\center{\includegraphics[width=7.5cm]{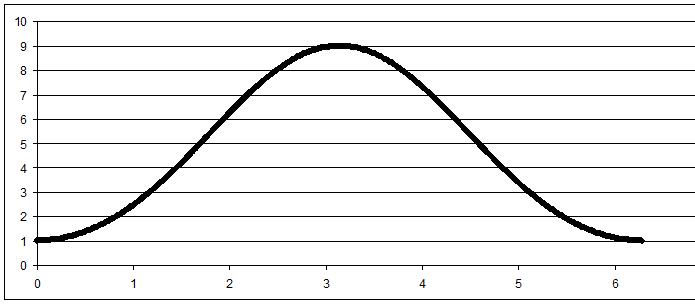}}\\
\centerline{a. $\varepsilon = 0,5$}
\end{minipage}
\begin{minipage}[h]{0.45\linewidth}
\center{\includegraphics[width=7.5cm]{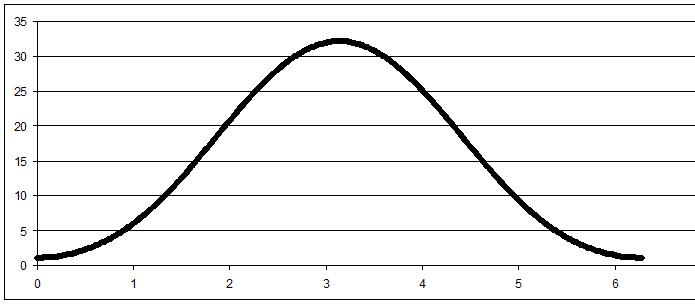}}\\%
\centerline{b. $\varepsilon = 0,7$}%
\end{minipage}
\caption{Graphs of the speed of the planet}\label{fig-09}
\end{figure}

\begin{figure}[p]
\begin{minipage}[h]{0.45\linewidth}
\center{\includegraphics[width=7.5cm]{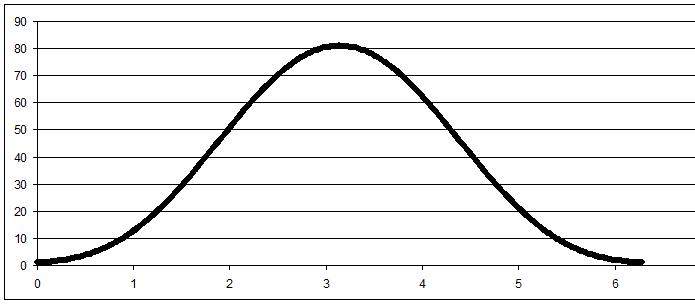}}\\
\centerline{a. $\varepsilon = 0,8$}
\end{minipage}
\begin{minipage}[h]{0.45\linewidth}
\center{\includegraphics[width=7.5cm]{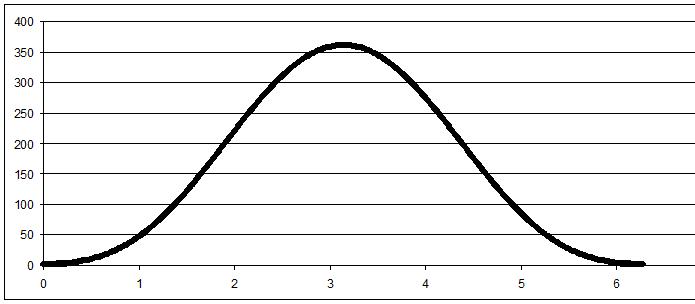}}\\%
\centerline{b. $\varepsilon = 0,9$}%
\end{minipage}
\caption{Graphs of the speed of the planet}\label{fig-10}
\end{figure}

\newpage

\section{Physical remarks about the Kepler's laws}

The notion of centripetal acceleration is well known from the
school course of physics. The fact is following: if a material
point of the mass $m$ moves on the circle of radius $R$ with
constant speed $v$, then the force $F = \frac{mv^2}{R}$ appears.

\subsection{Centripetal acceleration}

\subsubsection{Physical proof from the majority of textbooks}

Just for the completeness, we will present the proof of the
formula for the centripetal acceleration, which appear at the
majority of books.

\begin{figure}[h]
\center{\includegraphics[width=8cm]{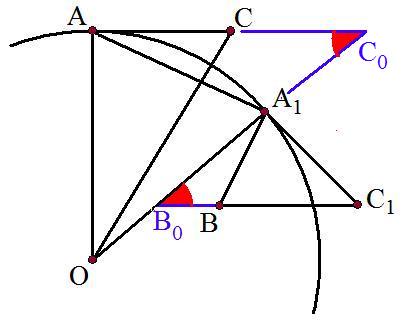}}%
\caption{}\label{fig-11}
\end{figure}

Write a circle (see Fig.~\ref{fig-11}) and let $AC$ be a vector of
the velocity at point $A$. Let $A_1$ be any another point of the
circle and $A_1C_1$ be the vector of velocity at this point. Thus,
$\overrightarrow{AC} = \overrightarrow{v}(t)$ and
$\overrightarrow{A_1C_1} = \overrightarrow{v}(t+\Delta t)$.

Take point $B$ such that $AC = BC_1$. Continue $AC$ and $BC_1$ to
obtain points $C_0$ and $B_0$ at the intersection with $AC$ and
$OA_1$ respectively. Clearly, $\angle C_0 = \angle B_0$, because
$AC$ is parallel to $BC_1$ by the construction. From triangles
$OAC_0$ and $B_0A_1C_1$ obtain that $$ \angle AOA_1 = 90^0 -
\angle C_0 = 90^0 - \angle B_0 = \angle A_1C_1B_0 = \angle
A_1C_1B,$$ whence $$ \angle AOA_1 = \angle A_1C_1B.$$

Thus, $\triangle AOA_1 \sim \triangle BC_1A_1$, because they are
isosceles with equal angles between equal sides. From their
similarity obtain $$ \frac{A_1C_1}{A_1O} = \frac{A_1B}{A_1A}.
$$ Plug physical values and obtain $$
\frac{v}{R} = \frac{\Delta v}{\Delta r} \Rightarrow \Delta v =
\frac{v\Delta r}{R}.
$$ Now, $$a = \frac{\Delta v}{\Delta t} = %
\frac{v}{R}\cdot \frac{\Delta v}{\Delta t} = \frac{v^2}{R}$$ and
we are done.

\subsubsection{Physical proof without triangles}

When a point moves by circle of radius $R$, the vector of its
velocity makes the entire rotation (rotation by $360^0$) around
the origin (we we consider the origin as the start points of the
vector.

Thus, the change of the vector of the velocity equals to the
length of the circle of the radius $R$, i.e. $2\pi |v|$. The time,
which is necessary for this rotation of the time of the rotation
of a point, which moves with the velocity $v$ by circle of radius
$R$, i.e. $t = \frac{2\pi R}{|v|}$. Thus, the acceleration is $$ a
=\frac{2\pi\, |v|}{ \frac{2\pi R}{|v|}} = \frac{v^2}{R}.
$$

\subsubsection{Pure mathematical proof}

Suppose that a point moves with a constant speed over a circle of
radius $R$ and center at origin. Then Cartesian coordinates of
this movement will be \begin{equation}\label{eq-03} \left\{
\begin{array}{l} x = R\, \cos
\omega t,\\
y = R\, \sin \omega t.
\end{array}\right.
\end{equation} The constant $\omega$ here is called radial speed and
determines, how quickly the point moves on the circle.

Since the period of functions $\sin$ and $\cos$ is $2\, \pi$, then
the period of the rotation of our point can be calculated from the
equation $ \omega T = 2\, \pi, $ whence $$ T = \frac{2\,
\pi}{\omega}.
$$ Since the length of the circle is $2\, \pi\, R$, then the
constant speed is $$ |\overrightarrow{v}| = \frac{2\, \pi\, R}{T}
= \omega\, R.
$$

Find the velocity vector as the derivative of~\eqref{eq-03}.
$$\left\{
\begin{array}{l} x' = -\omega\, R\, \sin
\omega t,\\
y' = \omega\, R\, \cos \omega t.
\end{array}\right.
$$ and acceleration will be $$\left\{
\begin{array}{l} x'' = -\omega^2\, R\, \cos
\omega t,\\
y'' = -\omega^2\, R\, \sin \omega t.
\end{array}\right.
$$ Thus, $$
|\overrightarrow{a}| = \omega^2\, R = \frac{v^2}{R}
$$ and this result is the same as the results, which were obtained
in previous sections.

\subsection{Forces, which act on a planet during its movement}

Thus, we know, that the formula for the centripetal acceleration
of the moving on the circle with constant speed is nothing more
that the law of the change of the vector of the velocity.

Clearly, the came, quite simple answer can be found by reasonings,
which look in such way, like mathematical is not used there. From
another hand it is clear, that ``mathematics is used'' in
``physical reasonings'' too, this mathematics is almost the same
and, in fact, is nothing more that the second derivative.

We will understand in this section the meaning and role of the
centripetal acceleration in the moving of a planet by its
elliptical orbit, assuming that this orbit is not a circle.

We will try to understand: is there any ``centripetal
acceleration''? What is its direction and what is its value?

When we worked with equation of the movement of a planet over its
orbit, we have the acceleration of the planet. We supposed that
this acceleration is cased by a force of Gravity, which is
directed to the Sun and, in fact, is directed to one of the
focuses of the ellipse. Clearly, there is no other acceleration of
the planet, i.e. the equation of the movement of the planet is an
exact function and this function has some exact second derivative,
which is the acceleration and determines the force, which acts on
the planet.

Nevertheless, we will find now some ``another'' acceleration.

Our idea will be the following. We will consider an ellipse to be
``locally a circle''. This will mean, that is we look ``at some
point'' and an ellipse and look at its ``neighbor'' points, then
they look like a segment of a circle. This circle (the radius and
the center) changes from point to point, but each point ``has
its'' circle. We will explain this idea a bit later, but suppose
that this is true and such circles exist. Is any movement on the
ellipse is considered as a movement by a circle, and speed of the
point (planet) and each position is found, we can find the
centripetal acceleration and can try to understand, which it will
be.

Lets come back the the understanding of the ellipse as ``local
circles''.

\begin{fact}\label{fact-24}
The projection of the acceleration of the point, whose movement is
described by vector-equation $s(t)$, to the normal to the velocity
at a point $t_0$ equals $$ |a_\nu| = \frac{v^2}{R},
$$ where $v = s'(t_0)$ and $R$ is from Fact~\ref{fact-23}.
\end{fact}

\newpage
\section{Detailed mathematical computations}\label{sect-proves}

\subsection{Derivatives and differentials}

\subsection{Polar coordinates}

\begin{fact*}[Fact~\ref{fact-03}]
Suppose that the position-vector of a point is given by polar
coordinates~\eqref{eq-02}. Then speed of this point can be found
by $$v^2 = \left( \frac{\Dif r}{\Dif t}\right)^2 +r^2\left(
\frac{d\theta}{dt}\right)^2
$$
\end{fact*}

\begin{proof}
Let $(r(t),\, \theta(t))$ be the equation of the movement in polar
coordinates. This means that $x(t) = r(t)\cos\theta(t)$ and $y(t)
= r(t)\sin\theta(t)$ is the law of the movement in Cartesian
coordinates. Thus, we cay calculate $v^2 = (x')^2 + (y')^2$ as
follows.

$$
x'(t) = r'(t)\cos \theta(t) -r(t)\theta'(t)\sin \theta(t);
$$$$y'(t) = r'(t)\sin\theta(t)
+r(t)\theta'(t)\cos\theta(t)).$$ Thus
$$ v^2(t) = (r'(t)\cos
\theta(t) -r(t)\theta'(t)\sin \theta(t))^2 +$$ $$+
(r'(t)\sin\theta(t) +r(t)\theta'(t)\cos\theta(t)))^2 =
$$$$
= \underline{(r'(t))^2\cos^2 \theta(t)} +
\underline{\underline{(r(t))^2(\theta'(t))^2\sin^2 \theta(t)}} -
$$ $$- \underline{\underline{\underline{2r(t)r'(t)\theta'(t)\sin
\theta(t)\cos \theta(t)}}} + \underline{(r'(t))^2\sin^2 \theta(t)}
+$$ $$+ \underline{\underline{(r(t))^2(\theta'(t))^2\cos^2
\theta(t)}} +
\underline{\underline{\underline{2r(t)r'(t)\theta'(t)\sin
\theta(t)\cos \theta(t)}}} =
$$ $$
= \underline{(r'(t))^2} +
\underline{\underline{(r(t))^2(\theta'(t))^2}} +
\underline{\underline{\underline{0}}} =
$$ $$ (r'(t))^2 +
(r(t))^2(\theta'(t))^2.
$$
This finishes the proof.
\end{proof}

\subsection{Integration}

\begin{fact*}[Fact~\ref{fact-16}]
The solution of the differential equation $$ \frac{f(x)}{g(y)} =
\frac{\Dif y}{\Dif x}
$$ can be written in the form
$$\int f(x)\Dif x - \int g(y)\Dif y = C(x),$$ where
$C'(x) = 0$.
\end{fact*}

Suppose that we have an equation \begin{equation}\label{eq-65}
\frac{f(x)}{g(y)} = \frac{\Dif y}{\Dif x},
\end{equation} where $y$ is considered as unknown function on $x$,
which has to be found, whenever functions $f$ and $g$ are given.

Suppose that we know functions $F$ and $G$ such that $F' = f$ and
$G' = g$. Then $$ \frac{\Dif (F(x) - G(y))}{\Dif x} = F'(x) -
G'(y)\cdot \frac{\Dif y}{\Dif x} = f(x) -g(y)\cdot \frac{\Dif
y}{\Dif x}.
$$ Now~\eqref{eq-65} implies $$f(x) -g(y)\cdot
\frac{\Dif y}{\Dif x} =0,$$ whence
\begin{equation}\label{eq-66}\frac{\Dif (F(x) - G(y))}{\Dif x} =0.
\end{equation} The equation~\eqref{eq-66} mens that $F(x) - G(y(x))$
is such function, whose derivative as the function of $x$ is zero.
The most simple example of such function is zero function (i.e.
which is zero everywhere), or constant function, or
piecewise-constant function.

More formally, we can say that the solution of the
equation~\eqref{eq-65} is $$\int f(x)\Dif x - \int g(y)\Dif y =
C(x),$$ where $C'(x) = 0$.

\subsection{Ellipse}

An ellipse is a plane geometrical figure, which is defined as
follows. Fix two points, say $F_1,\, F_2$ on the plain, which will
contain the ellipse. The ellipse is consisted of all the points
$M$ of the plain such that \begin{equation}\label{eq-71} MF_1 +
MF_2 = 2a,\end{equation} where $a$ is fixed at the very beginning.
Points $F_1$ and $F_2$ are called focuses of the ellipse.

\begin{fact*}[Fact~\ref{fact-20}]
If the Cartesian coordinates of the focuses of the ellipse are
$F_1(-f,\, 0)$ and $F_2(f,\, 0)$, then $$ \frac{x^2}{a} +
\frac{y^2}{b} = 1
$$ is an equation of the ellipse.
\end{fact*}

\begin{proof}

Denote the cartesian coordinates of focuses of an ellipse by
$F_1(-f,\, 0)$ and $F_2(f,\, 0)$, whence obtain the equation of
this ellipse in cartesian plain. Rewrite~\eqref{eq-71} as $$
\sqrt{(x+f)^2+y^2} + \sqrt{(x-f)^2+y^2} = 2a;
$$ $$
(x+f)^2+y^2 = (2a -\sqrt{(x-f)^2+y^2})^2;
$$
$$
\underline{x^2} +2xf +\underline{\underline{f^2 +y^2}} = 4a^2
-4a\sqrt{(x-f)^2+y^2} + \underline{x^2} -2xf +
\underline{\underline{f^2 +y^2}};
$$
$$
2xf  = 4a^2 -4a\sqrt{(x-f)^2+y^2}  -2xf;
$$
$$
a\sqrt{(x-f)^2+y^2} = a^2 -xf;
$$
$$
a^2x^2 -\underline{2ax^2xf} +a^2f^2 +a^2y^2 = a^4 -
\underline{2a^2xf} + x^2f^2;
$$
$$
x^2(a^2 -f^2) +a^2y^2 = a^2(a^2 -f^2);
$$
\begin{equation}\label{eq-07}
\frac{x^2}{a^2} + \frac{y^2}{a^2 -f^2} = 1.
\end{equation}

If follows from the equality of the triangle $F_1F_2\leq MF_1
+MF_2$ that $2f\leq 2a$. Thus, denote
\begin{equation}\label{eq-09}b^2 = a^2 - f^2,\end{equation}
and rewrite~\eqref{eq-07} as
\begin{equation}\label{eq-08} \frac{x^2}{a^2} + \frac{y^2}{b^2} =
1.
\end{equation} \end{proof}

Equation~\eqref{eq-08} is called %
the \textbf{canonical equation of an ellipse} in Cartesian form.

Due to~\eqref{eq-08} we can imagine an ellipse as an oval figure,
which is bounded by lines $x = \pm a$ and $y=\pm b$. If follows
from~\eqref{eq-09} that $a\geq b$ and
\begin{equation}\label{eq-10}f = \sqrt{a^2-b^2}.
\end{equation}

The expression \begin{equation}\label{eq-11} \varepsilon =
\frac{f}{a}
\end{equation} is called the %
\textbf{eccentricity} and characterizes the differ of the ellipse
and a circle. Notice, that $\varepsilon\in [0,\, 1)$. Moreover,
$\varepsilon =0$ if and only if an ellipse is the circle.
Controversially, if $\varepsilon\approx 1$, then an ellipse
transforms to a figure, which is proximate to an line segment
between points $(-a,\, 0)$ and $(a,\, 0)$.

\begin{fact*}[Fact~\ref{fact-21}]
If $F_1$ is in the polar origin and $F_2$ is on the polar axis,
then the equation of the ellipse is $$ r =
\frac{p}{1-e\cos\theta}, $$ where $$ p = \frac{b^2}{a}$$ and is
called semi-latus rectum of the ellipse.
\end{fact*}

\begin{proof}
Put the origin to the left focus $F_1$ and the $X$-axis (and polar
axis) out through the second focus $F_2$. Then cartesian
coordinates of $F_2$ would be $(2c,\, 0)$.

Let $M$ be a ``general point''of the ellipse ad let $(r,\,
\theta)$ be its polar coordinates. Then cartesian coordinates of
$M$ are $M(r\cos\theta,\, r\sin\theta)$. Then
$$ MF_2 = \sqrt{(2f)^2 +r^2\cos^2\theta
-4rf\cos\theta +r^2\sin^2\theta} =
$$$$= \sqrt{r^2 -4rf\cos\theta +4f^2}
$$ and we can rewrite~\eqref{eq-71} as
 $$ r +
\sqrt{r^2 -4rf\cos\theta +4f^2} = 2a.
$$
Subtract $r$ from both sides, square this equality and obtain $$
r^2 -4rf\cos\theta +4f^2 = 4a^2 +r^2 -4ar.
$$
After the evident simplification get
$$
-rf\cos\theta +f^2 = a^2 -ar.
$$
Now express $r$
$$
r =\frac{a^2 -f^2}{a-f\cos\theta} = \frac{a^2
-f^2}{a\left(1-\frac{f}{a}\cos\theta\right)},
$$ and we are done.
\end{proof}

If follows from~\eqref{eq-09} and~\eqref{eq-11} that $$ r =
\frac{p}{1-\varepsilon\cos\theta},
$$ where $p=\frac{b^2}{a}$. %
Notice, that $p$ is called \textbf{semi-latus rectum} of the
ellipse.

\begin{fact*}[Fact~\ref{fact-22}]
The area of the ellipse can be calculated as $A = \pi\, ab.$
\end{fact*}

\begin{proof}
Before the beginning of the proof of this fact we should give the
explanation of the notion of the area. It is a well known fact
that area of the rectangle is the product of its sides. The same
is true for the rectangle with sides $\Delta x$ and $\Delta y$,
which are ``very small'' and are parallel to $X$-axis and $Y$-axis
respectively. The area of any geometrical is defines as the sum of
all squares $\Delta x \times \Delta y$, which belong to our figure
and such that $\Delta x \approx 0$ and $\Delta y \approx 0$(in
fact $\Delta x = \Dif x$, $\Delta y = \Dif y$ and the area is some
integral, but we will not come into these details).

Suppose that $a = b$ and our ellipse is a circle. Clearly, in this
case its area is $\pi b^2$. Now we sketch in the $X$-axis
direction the big square, where the circle is inscribed in such
way that the width of the appeared rectangle will become $a$, i.e.
we sketch $\frac{a}{b}$ times. The equation of the new figure will
be $$\frac{x^2}{a^2} + \frac{y^2}{b^2} = 1,$$ but it is exactly
the equation~\eqref{eq-08} of the ellipse.

The area of the new geometrical figure can be calculates follows.
Since each rectangle $\Delta x \times \Delta y$ is transformed to
$\frac{a}{b}\cdot \Delta x\times \Delta y$, then the former area
$\pi b^2$ is multiplied by $\frac{a}{b}$ and becomes equal to
$\pi\, ab$.
\end{proof}

\subsection{Mathematical expression of the Second Kepler's law}

\begin{fact*}[Fact~\ref{fact-19}] Let the trajectory of the planet is
given in Cartesian coordinates as $$ \frac{x^2}{a^2} +
\frac{y^2}{b^2} = 1,
$$ where $a>b>0$. Then the Second Kepler's law can be rewritten as
$$
x\frac{dy}{dt} -y\frac{dx}{dt} = C,
$$ %
where $C$ is a constant, dependent on the planet and independent
on time. Moreover, $$C = \frac{A}{T}, $$ where $A$ is the area of
the ellipse, which is the trajectory of the planet and $T$ is the
period of the rotation.
\end{fact*}

\begin{proof}
Denote $r(t)=(x(t),\, y(t))$ the coordinates of the planet of the
moment $t$. By Fact~\ref{fact-18} (above in this section), the
area between vectors $\overrightarrow{r}(t)$ and
$\overrightarrow{r}(t+\Delta t)$ equals
\begin{equation}\label{eq-49}A(t,\, t+\Delta t) = \abs\left(
x(t)y(t+\Delta t) - y(t)x(t+\Delta t)\right).\end{equation}

Second Kepler's Law claims that there is a constant $C$ such that
\begin{equation}\label{eq-46}
A(t,\, t+\Delta t) = C\, \Delta t.\end{equation} For the entire
ellipse (trajectory of a planet) denote $A$ its area and $T$ the
period of the movement we, evidently, have $$ A = CT,
$$ whence \begin{equation}\label{eq-47}
C = \frac{A}{T}.
\end{equation}
Since the direction (clockwise or, anti clockwise) of the movement
of a planet is constant (it trivially follows from the Second
Kepler's law), then we can remove the absolute value sign
in~\eqref{eq-49} and obtain from~\eqref{eq-46} and~\eqref{eq-47}
that \begin{equation}\label{eq-48}\left( x(t)y(t+\Delta t) -
y(t)x(t+\Delta t)\right) = \frac{A}{T}\, \Delta t.\end{equation}

Divide both sides by $\Delta t$ and consider $\Delta t \approx 0$,
i.e. take the difference $\Dif t$ instead of $\Delta t$, and
obtain
\begin{equation}\label{eq-4-7-a} x\frac{dy}{dt} -y\frac{dy}{dt} =
\frac{A}{T}.
\end{equation}
\end{proof}

\begin{fact*}[Fact~\ref{fact-04}]
If the Sun is in the Origin, then the Second Kepler's law in polar
coordinates is expressed as
$$
r^2\frac{d\theta}{dt}=C,
$$ where $C$ is the constant from Fact~\ref{fact-19}.
\end{fact*}

\begin{proof}
We will rewrite~\ref{eq-4-7-a} (the Second Kepler's law from
Fact~\ref{fact-19}) in Polar coordinates. If $$
\left\{\begin{array}{l}
x(t)=r(t)\cos (\theta(t))\\
y(t)=r(t)\sin(\theta(t)),
 \end{array}
\right.$$ then $$\left\{\begin{array}{l}
x'(t)=r'(t)\cos (\theta(t)) - r(t)\theta'(t)\sin\theta(t)\\
y'(t)=r'(t)\sin(\theta(t)) + r(t)\theta'(t)\cos\theta(t),
 \end{array}\right.
$$
and we can rewrite~\eqref{eq-4-7-a} as $$ C = r\cos \theta\cdot
(r'\sin\theta + r\theta'\cos\theta) - r\sin\theta\cdot (r'\cos
\theta - r\theta'\sin\theta) =
$$$$
r^2\theta'\cdot (\cos^2\theta + \sin^2\theta) =r^2\theta'
$$ and we are done.
\end{proof}

\begin{fact*}[Fact~\ref{fact-17}]
Denote $u = \frac{1}{r}$. Then in polar coordinates the Second
Kepler's Law can be written as $$ v^2 = C^2\left[\left(\frac{\Dif
u}{\Dif t}\right)^2 +u^2\right].
$$
\end{fact*}

\begin{proof}
Notice, that, by Fact~\ref{fact-03}, we have
\begin{equation}\label{eq-69}
v^2 =\left( \frac{\Dif r}{\Dif t}\right)^2 +r^2\left(
\frac{d\theta}{dt}\right)^2. \end{equation}

If follows from Fact~\ref{fact-04} that $$ \frac{\Dif \theta}{\Dif
t} = \frac{C}{r^2},
$$ whence $$\frac{\Dif r}{\Dif t} = \frac{\Dif r}{\Dif
\theta}\cdot \frac{\Dif \theta}{\Dif t} = \frac{C}{r^2}\cdot
\frac{\Dif r}{\Dif \theta}.
$$ Thus, we can rewrite~\eqref{eq-69} as
$$v^2 = C^2\left[ \left( \frac{1}{r^2}\,
\frac{dr}{d\theta}\right)^2 +\frac{1}{r^2} \right].
$$
Denote $$u = \frac{1}{r},$$ whence $$ \frac{\Dif u}{\Dif t} =
\frac{-1}{r^2}\cdot \frac{\Dif r}{\Dif t}
$$ and we are done.
\end{proof}

\subsection{Equivalence of the Third Kerpler's Law and the Fourth
Newton's law}\label{sect-04}

We will present in this Section the proof of Fact~\ref{fact-08}.

\begin{fact*}[Fact~\ref{fact-08}]
Suppose that the Sun is fixed in the Origin and a planet moves by
the elliptic orbit with polar equation
\begin{equation}\label{eq-32}%
r = \frac{p}{1-e\, \cos\theta}\end{equation} %
and the equation $$ r^2\, \frac{\Dif \theta}{\Dif t} = C
$$ holds. Then the acceleration of the planet is $$
\overrightarrow{a} = -\,
\frac{\overrightarrow{r}}{|\overrightarrow{r}|}\cdot
\frac{C^2}{p\, r^2}.
$$
\end{fact*}

\begin{proof}
We will find the vector of the acceleration by differentiating %
$$ \left\{\begin{array}{l}x = r\cos\theta\\
y = r\sin\theta.
\end{array}\right.
$$ two times as follows: %
$$ \left\{\begin{array}{l}x' = r'\cos\theta -r\theta'\sin\theta\\
y' = r'\sin\theta +r\theta'\cos\theta.
\end{array}\right.
$$\begin{equation}\label{eq-34}
\left\{\begin{array}{l}x'' = r''\cos\theta -2r'\theta'\sin\theta
-r\theta''\sin\theta
-r\theta'^2\cos\theta\\
y'' = r''\sin\theta +2r'\theta'\cos\theta +r\theta''\cos\theta
-r\theta'^2\sin\theta.
\end{array}\right.
\end{equation} %
The derivative of $r$ can be found from~\eqref{eq-32} as $$ r'=
\frac{-p\, \varepsilon\sin\theta}{(1-\varepsilon\cos\theta)^2}\,
\theta'.
$$ Notice, that if follows from~\eqref{eq-32} that
\begin{equation}\label{eq-33}
\theta' = \frac{C}{r^2}.
\end{equation} Thus, we can simplify the derivative $r'$ as
$$ r'
=\frac{-C\varepsilon\sin\theta}{p}.
$$ Next, $$
r'' =\frac{-C\varepsilon\cos\theta}{p}\, \theta' $$
and~\eqref{eq-33} implies
$$r''= \frac{-C^2\varepsilon\cos\theta(1 -
\varepsilon\cos\theta)^2}{p^3}.
$$ Again by~\eqref{eq-33} find $\theta''$ as $$
\theta'' = \frac{2\, C\varepsilon\sin\theta}{p^2}\,
(1-\varepsilon\cos \theta)\theta',
$$ whence $$
\theta'' =\frac{2\, C^2\varepsilon\sin\theta}{p^4}\,
(1-\varepsilon\cos \theta)^3.
$$ Using expressions for $r'$, $r''$, $\theta'$ and %
$\theta''$, we can rewrite the first line of~\eqref{eq-34} as $$
x'' = r''\cos\theta -2r'\theta'\sin\theta -r\theta''\sin\theta
-r\theta'^2\cos\theta =$$$$ =\frac{-C^2\varepsilon\cos^2\theta(1 -
\varepsilon\cos\theta)^2}{p^3} +2\cdot
\frac{C\varepsilon\sin\theta}{p}\cdot \frac{C(1-\varepsilon\cos
\theta)^2}{p^2}\, \, \sin\theta -$$
$$-\frac{p}{1-\varepsilon\cos\theta}\cdot \frac{2\, C^2\varepsilon\sin\theta}{p^4}\,
(1-\varepsilon\cos \theta)^3\, \sin\theta
-\frac{p}{1-\varepsilon\cos\theta}\cdot \frac{C^2}{p^4}\,
(1-\varepsilon\cos \theta)^4\, \cos\theta =$$
$$
=\frac{-C^2(1-\varepsilon\cos \theta)^2}{p^3}\, \left(\varepsilon
\cos^2\theta -2\varepsilon \sin^2\theta +2\varepsilon \sin^2\theta
+(1-\varepsilon \cos\theta)\cos\theta\right)=
$$$$
=\frac{-C^2(1-\varepsilon\cos \theta)^2}{p^3}\, \cos\theta\, .
$$
Analogously, the second line of~\eqref{eq-34} can be simplified as
$$
y'' =r''\sin\theta +2r'\theta'\cos\theta +r\theta''\cos\theta
-r\theta'^2\sin\theta =
$$$$ = \frac{-C^2\varepsilon\cos\theta(1 -
\varepsilon\cos\theta)^2}{p^3}\, \sin\theta +2\cdot
\frac{-C\varepsilon\sin\theta}{p}\cdot \frac{C}{p^2}\,
(1-\varepsilon\cos \theta)^2\, \cos\theta +$$
$$+\frac{p}{1-\varepsilon\cos\theta}\cdot \frac{2\, C^2\varepsilon\sin\theta}{p^4}\,
(1-\varepsilon\cos \theta)^3\, \cos\theta
-\frac{p}{1-\varepsilon\cos\theta}\cdot \frac{C^2}{p^4}\,
(1-\varepsilon\cos \theta)^4\, \sin\theta =$$
$$=\frac{-C^2(1-\varepsilon\cos
\theta)^2\sin\theta}{p^3}\left(\varepsilon\cos\theta
+2\varepsilon\cos\theta -2\varepsilon\cos\theta +
(1-\varepsilon\cos\theta)\right)=
$$$$
=\frac{-C^2(1-\varepsilon\cos \theta)^2\sin\theta}{p^3}.
$$ Thus, we obtain formulas for $a$
as
$$\left\{\begin{array}{l}x'' = \frac{-C^2\, \cos\theta}{p\, r^2},\\
y'' = \frac{-C^2\, \sin\theta}{p\, r^2}.
\end{array}\right.
$$
\end{proof}

\subsection{Analytical representation of the movement}\label{sect-4}

We will prove

\begin{fact*}[Fact~\ref{fact-14}]
Suppose that $\varepsilon\in (0,\, 1)$. Then $$\int \frac{\Dif
x}{(1-\varepsilon \cos x)^2} = 
\frac{\sqrt{1-\varepsilon}}{\sqrt{1+\varepsilon}}\cdot
\frac{2}{(1-\varepsilon^2)(1-\varepsilon)} \left( \arctan t
 + \frac{t\, \varepsilon}{t^2+1} \right),
$$
where $t = \tan\left( \frac{x}{2}\cdot
\frac{\sqrt{1+\varepsilon}}{\sqrt{1-\varepsilon}}\right).$
\end{fact*}

We will need two additional facts for the proof of
Fact~\ref{fact-14}:

\begin{fact}\label{fact-06}
$$\int \frac{\Dif x}{\left(
x^2+1\right)^2} =\frac{\arctan x}{2} + \frac{x}{2x^2+2} +C$$
\end{fact}

\begin{proof}
Differentiate the expression $\frac{x}{x^2+1}$ and obtain $$
\frac{\Dif}{\Dif x}\left( \frac{x}{x^2+1} \right) =
\frac{1}{x^2+1} - \frac{2x^2}{(x^2+1)^2} = \frac{1}{x^2+1} -
2\left (\frac{1}{x^2+1} - \frac{1}{(x^2+1)^2}\right) =
$$$$
=\frac{-1}{x^2+1} +\frac{2}{(x^2+1)^2}.
$$

Thus, taking integrals from the former and the last part of this
equality obtain
$$
\frac{x}{x^2+1} = - \int \frac{1}{x^2+1}\Dif x + 2\int
\frac{1}{x^2+1}\Dif x,
$$ whence $$
\int \frac{1}{x^2+1}\Dif x = \frac{1}{2}\left( \arctan x
+\frac{x}{x^2+1}\right).
$$
\end{proof}

\begin{fact}\label{fact-07}
$$\int \frac{x^2\Dif x}{\left(
x^2+1\right)^2} =\frac{\arctan x}{2} - \frac{x}{2x^2+2} +C$$
\end{fact}

\begin{proof}
$$\int \frac{x^2\Dif x}{\left( x^2+1\right)^2} = \int \left( \frac{1}{x^2+1} -\frac{1}{\left( x^2+1\right)^2}\right)\Dif x.$$
Applying Fact~\ref{fact-06}, obtain $$\int \frac{x^2\Dif x}{\left(
x^2+1\right)^2} = \arctan x -\frac{1}{2}\left( \arctan x
+\frac{x}{x^2+1}\right)
$$ obtain the necessary.
\end{proof}

Now we are ready to prove Fact~\ref{fact-14}.

\begin{proof}[Proof of Fact~\ref{fact-14}]

Denote $\Theta = \frac{1}{(1-\varepsilon\cos \theta)^2},$ $J =
\frac{\Dif \theta}{(1-\varepsilon\cos \theta)^2} $, and $I = \int
J.$

Denote $$ \psi = \tan\frac{\theta}{2}.
$$ Then $$
\cos\theta = \frac{1-\psi^2}{1+\psi^2}
$$ and $$
d\theta = \frac{2d\psi}{1+\psi^2}
$$ Thus, $$
\frac{\Dif \theta}{(1-\varepsilon\cos \theta)^2} = \frac{1}{\left(
1-\varepsilon\left(\frac{1-\psi^2}{1+\psi^2}\right)
\right)^2}\cdot \frac{2d\psi}{1+\psi^2}=
$$$$
= \frac{2(1+\psi^2)d\psi}{\left( (1+\psi^2) -
\varepsilon(1-\psi^2)\right)^2} = \frac{2(1+\psi^2)d\psi}{\left(
\psi^2(1+\varepsilon) + (1-\varepsilon)\right)^2}.
$$

We should notice, that the obtained expression is ``similar'' to
examples, which are given in Facts~\ref{fact-06}
and~\ref{fact-07}, but we need some additional transformations for
the direct use of these facts.

$$
\frac{2(1+\psi^2)d\psi}{\left( \psi^2(1+\varepsilon) +
(1-\varepsilon)\right)^2} = \frac{1}{(1-\varepsilon)^2}\cdot
\frac{2(1+\psi^2)d\psi}{\left(
\left(\frac{\psi\sqrt{1+\varepsilon}}{\sqrt{1-\varepsilon}}
\right)^2 + 1\right)^2} =
$$$$
= \frac{1}{(1-\varepsilon)^2}\cdot
\frac{2\left(\frac{1+\varepsilon}{1-\varepsilon}
+\left(\frac{\psi\sqrt{1+\varepsilon}}{\sqrt{1-\varepsilon}}
\right)^2\right)d\psi}{\left(
\left(\frac{\psi\sqrt{1+\varepsilon}}{\sqrt{1-\varepsilon}}
\right)^2 + 1\right)^2}\cdot \frac{1-\varepsilon}{1+\varepsilon} =
$$$$
= \frac{1}{1-\varepsilon^2}\cdot
\frac{2\left(\frac{1+\varepsilon}{1-\varepsilon}
+\left(\frac{\psi\sqrt{1+\varepsilon}}{\sqrt{1-\varepsilon}}
\right)^2\right)d\psi}{\left(
\left(\frac{\psi\sqrt{1+\varepsilon}}{\sqrt{1-\varepsilon}}
\right)^2 + 1\right)^2}.
$$ Denote $\zeta =
\frac{\psi\sqrt{1+\varepsilon}}{\sqrt{1-\varepsilon}}$. Then
$d\psi = \frac{\sqrt{1-\varepsilon}}{\sqrt{1+\varepsilon}}\,
d\zeta$, whence $$ I = \frac{1}{1-\varepsilon^2}\cdot
\frac{2\left(\frac{1+\varepsilon}{1-\varepsilon}
+\zeta^2\right)}{\left( \zeta^2 + 1\right)^2}\cdot
\frac{\sqrt{1-\varepsilon}}{\sqrt{1+\varepsilon}}\, d\zeta =
$$
$$= \frac{2\sqrt{1-\varepsilon}}{\sqrt{1+\varepsilon}}\cdot \frac{1}{1-\varepsilon^2}
\left( \frac{1+\varepsilon}{1-\varepsilon}\cdot
\frac{d\zeta}{\left( \zeta^2 + 1\right)^2} + \frac{ \zeta^2\,
d\zeta}{\left( \zeta^2 + 1\right)^2}\, \right).
$$ Using Facts~\ref{fact-06}
and~\ref{fact-07} simplify $$ \int I
=\frac{2\sqrt{1-\varepsilon}}{\sqrt{1+\varepsilon}}\cdot
\frac{1}{1-\varepsilon^2} \left(
\frac{1+\varepsilon}{1-\varepsilon}\cdot \left(\frac{\arctan
\zeta}{2} + \frac{\zeta}{2\zeta^2+2}\right) + \left(\frac{\arctan
\zeta}{2} - \frac{\zeta}{2\zeta^2+2}\right) \right)=
$$$$
= \frac{\sqrt{1-\varepsilon}}{\sqrt{1+\varepsilon}}\cdot
\frac{1}{1-\varepsilon^2} \left(
\frac{1+\varepsilon}{1-\varepsilon}\cdot \left(\arctan \zeta +
\frac{\zeta}{\zeta^2+1}\right) + \left(\arctan \zeta -
\frac{\zeta}{\zeta^2+1}\right) \right)=
$$$$
= \frac{\sqrt{1-\varepsilon}}{\sqrt{1+\varepsilon}}\cdot
\frac{1}{1-\varepsilon^2} \left( \arctan
\zeta\left(\frac{1+\varepsilon}{1-\varepsilon}+1\right) +
\frac{\zeta}{\zeta^2+1}\left(\frac{1+\varepsilon}{1-\varepsilon}-1
\right)\right)=
$$$$
= \frac{\sqrt{1-\varepsilon}}{\sqrt{1+\varepsilon}}\cdot
\frac{1}{1-\varepsilon^2} \left( \arctan \zeta \cdot
\frac{2}{1-\varepsilon} + \frac{\zeta}{\zeta^2+1}\cdot
\frac{2\varepsilon}{1-\varepsilon} \right) =
$$$$
= \frac{\sqrt{1-\varepsilon}}{\sqrt{1+\varepsilon}}\cdot
\frac{2}{(1-\varepsilon^2)(1-\varepsilon)} \left( \arctan \zeta
 + \frac{\zeta\, \varepsilon}{\zeta^2+1} \right).
$$

Notice, that since $\psi = \tan \frac{\theta}{2}$ and $\zeta =
\frac{\psi\sqrt{1+\varepsilon}}{\sqrt{1-\varepsilon}}$, then
$\zeta = \tan\left( \frac{\theta}{2}\cdot
\frac{\sqrt{1+\varepsilon}}{\sqrt{1-\varepsilon}}\right).$
\end{proof}

\subsection{Physical remarks about Kepler's laws}\label{sect-08}

\begin{fact*}[Fact~\ref{fact-24}]
The projection of the acceleration of the point, whose movement is
described by vector-equation $s(t)$, to the normal to the velocity
at a point $t_0$ equals $$ |a_\nu| = \frac{v^2}{R},
$$ where $v = s'(t_0)$ and $R$ is from Fact~\ref{fact-23}.
\end{fact*}

\begin{proof}
Remind that $R$ here is the radius of a circle, such that we can
``locally'' consider our curve as a circle with this radius. Thus,
``from physics'' obtain that $\overrightarrow{a}_c$ with $$a_c  =
\frac{(v(t_0))^2}{R}$$ is centripetal acceleration, which should
appear, where $v(t_0) = s'(t_0)$ is the velocity of a point, which
can be found as the derivative.

From another hand, it is clear that the tangent to our curve (to
the trajectory of the planet), which is parallel to the velocity
$s'(t_0)$ should also be the tangent to the circle, which
``locally'' is equaled to the curve. This means that the center of
this circle should belong to the perpendicular to the tangent at
$t_0$. Denote $\nu$ the direction vector of this perpendicular,
i.e. $\nu \cdot s'(t_0) =0$.

Find the projection $a_\nu$ of the acceleration
$\overrightarrow{a}$ of our point on $\nu$. If $\alpha$ is the
angle between $a$ and $\nu$, then it is clear that $$a_\nu = a\,
\cos \alpha.$$

Since $\nu$ is perpendicular to $v$, then $\beta =90^0-\alpha$ is
the angle between $a$ and $v$. Find the area of the parallelogram,
whose sides are $a$ and $v$. From one hand it is $|a|\cdot |v|\,
\sin\beta$. From another hand, it follows from Fact~\ref{fact-18}
that this area is $\mathcal{P}_3[\overrightarrow{a},\,
\overrightarrow{v}]|$. Thus, $$|a|\cdot |v|\, \sin\beta =
|\mathcal{P}_3[\overrightarrow{a},\, \overrightarrow{v}]|,$$
whence $$ \sin\beta = \cos\alpha =
\frac{|\mathcal{P}_3[\overrightarrow{a},\,
\overrightarrow{v}]|}{|s''(t_0)|\cdot |s'(t_0)|}
$$ and $$
|a_v| = |a|\, \cos\alpha =
\frac{|\mathcal{P}_3[\overrightarrow{a},\,
\overrightarrow{v}]|}{|s'(t_0)|}
$$
Notice, that this expression of exactly the same as $$
\frac{v^2}{R},
$$ where $v = s'(t_0)$ and $R$ is found above.
\end{proof}

\subsubsection{About vector product}\label{sect-09}

We have already mentioned in Fact~\ref{fact-18} that if we have
two vectors $a$ and $b$ of the plane XOY, then $$ A
=|\mathcal{P}_3[a,\, b]|
$$ is the area of the parallelogram, whose sides are vectors $a$ and
$b$. In fact, the projection $\mathcal{P}_3$ and the fact that
$a,$ and $b$ belong to the plane XOY is not necessary in
Fact~\ref{fact-18}. The more general fact is true.

\begin{fact}\label{fact-25}
For any $3$-dimensional vectors $a$ and $b$ the area of the
parallelogram with sides $a$ and $b$ equals $$ A = |[a,\, b]|,
$$ where $[a,\, b]$ is the vector product of $a$ and $b$.
\end{fact}

Clearly, we will explain now the definition of vector product and
will prove Fact~\ref{fact-25}.

Consider the $3$-dimensional space, and denote $$ i = (1,\, 0,\,
0),
$$$$
j = (0,\, 1,\, 0)
$$ and $$
k= (0,\, 0,\, 1).
$$
Define the vector product as follows:

1. $[i,\, j] = k$, $[j,\, k] = i$ and $[k,\, i] = j$.

2. $[a,\, b] = -[b,\, a]$ for all $a,\, b$.

3. $[\alpha a + \beta b,\, c] = \alpha [a,\, c] + \beta [b,\, c]$.

Now we are ready to prove Fact~\ref{fact-25}.

\begin{proof}[Proof of Fact~\ref{fact-25}]
Suppose that $a = (x_1,\, y_1,\, z_1)$ and $b = (x_2,\, y_2,\,
z_2)$. In other words $a = x_1i + y_1j + z_1k$ and $b =x_2i + y_2j
+ z_2k$. Then $$ [a,\, b] = [x_1i + y_1j + z_1k,\, x_2i + y_2j +
z_2k] =
$$$$
=x_1y_2k + x_1z_2(-j) + y_1x_2(-k) +y_1z_2i + z_1x_2j +z_1y_2(-i)
=
$$$$= (y_1z_2 -z_1y_2)i + (z_1x_2 -x_1z_2)j + (x_1y_2-y_1x_2)k.$$

Clearly, $$ a^2 = x_1^2 +y_1^2 + z_1^2,
$$ $$
b^2 = x_2^2 +y_2^2 + z_2^2
$$ and  $$
a\cdot b = x_1x_2 +y_1y_2 +z_1z_2.
$$ Denote $\alpha$ the angle between $a$ and $b$. %
Then $(a\cdot b)^2 = a^2\, b^2\, \cos^2\alpha$, whence $$ a^2\,
b^2\, \sin^2\alpha = a^2\, b^2 - (a\cdot b)^2.
$$ By direct calculation check that $$
a^2\, b^2 - (a\cdot b)^2 = [a,\, b]^2,
$$ $$
a\cdot [a,\, b] = 0
$$ and $$
b\cdot [a,\, b] = 0
$$ and we are done.
\end{proof}

\newpage
\section{Some information about Solar System}

Eccentricities of planets of Solar systems are given in table
below (see~\cite{Eccentr}).

$$
\begin{array}{ll}
\cline{1-2} \multicolumn{1}{|l|}{\text{Mercury}} & \multicolumn{1}{l|}{0.20563069} \\
\cline{1-2} \multicolumn{1}{|l|}{\text{Venus}}  & \multicolumn{1}{l|}{0.00677323} \\
\cline{1-2} \multicolumn{1}{|l|}{\text{Earth}}  &  \multicolumn{1}{l|}{0.01671022} \\
\cline{1-2} \multicolumn{1}{|l|}{\text{Mars}} & \multicolumn{1}{l|}{0.09341233} \\
\cline{1-2} \multicolumn{1}{|l|}{\text{Jupiter}} & \multicolumn{1}{l|}{0.04839266} \\
\cline{1-2} \multicolumn{1}{|l|}{\text{Saturn}} & \multicolumn{1}{l|}{0.05415060} \\
\cline{1-2} \multicolumn{1}{|l|}{\text{Uranus}} & \multicolumn{1}{l|}{0.04716771} \\
\cline{1-2} \multicolumn{1}{|l|}{\text{Neptune}} & \multicolumn{1}{l|}{0.00858587} \\
\cline{1-2}
\end{array}$$

We can see, that all of them, except Mercury, are less than $0.1$,
whence all the orbits are very close to circles.

\newpage
\pagestyle{empty}

\bibliography{Astron}{}
\bibliographystyle{plain}

\newpage

\tableofcontents

\end{document}